\documentclass[12pt]{scrartcl}                  
\usepackage{amsmath,amsthm,thm-restate,thmtools,amsfonts,amssymb}
\usepackage[numbers]{natbib}
\usepackage{hyperref,color,comment}

\usepackage{latexsym}
\date{}

\begin{document}

\newcommand{\out}{\ensuremath{\mathrm{Out}(\mathbb{F}) } }
\newcommand{\aut}{\ensuremath{\mathrm{Aut}(\mathbb{F}) }}
\newcommand{\F}{\ensuremath{\mathbb{F} } }
\newcommand\addtag{\refstepcounter{equation}\tag{\theequation}}
\newtheorem{theorem}{Theorem}[section]
\newtheorem{corollary}[theorem]{Corollary}
\newtheorem{lemma}[theorem]{Lemma}
\newtheorem{proposition}[theorem]{Proposition}
\newtheorem{remark}[theorem]{Remark}
\newtheorem{definition}[theorem]{Definition}
\newtheorem*{thma}{Theorem A}

\title{Relatively irreducible free subgroups in $\out$}



\author{Pritam Ghosh}



\maketitle

\begin{abstract}
We construct examples of free-by-free extensions which are (strongly) relatively hyperbolic. For this 
we consider certain class of exponentially growing outer automorphisms, which are not fully irreducible themselves but behave like 
fully irreducibles in the complement of a free factor system $\mathcal{F}$. We construct a free group using two \emph{independent} elements 
of this type and prove that this gives us a free-by-free relatively hyperbolic extension. This generalizes similar results obtained in the case of 
surface group with punctures.

  \vspace{0.5cm}


\end{abstract}

\section{Introduction}
\label{intro}

Let $\F$ be a free group of rank $N\geq 3$. The quotient group 
$\aut/\mathrm{Inn}(\F)$, denoted by $\out$,  is called the group  of outer automorphisms of $\F$. There are many tools in studying the properties of this group. One of 
them is by using train-track maps introduced Bestvina-Handel \cite{BH-92} and later generalized by Bestvina-Feighn-Handel \cite{BFH-97}, \cite{BFH-00} and Feighn-Handel \cite{FH-11}. 
$\out$ admits an action on the set of conjugacy classes of free factors of $\F$. An element $\phi$ is said to be \emph{fully irreducible} if 
there are no $\phi$-periodic conjugacy class of any free factor.
The fully-irreducible outer automorphisms are the most well understood elements in $\out$ . 
They behave very closely to the pseudo-Anosov homeomorphisms of surfaces with one boundary component, 
which have been well understood and are a rich source of examples and interesting theorems. We however, 
will focus on exponentially growing outer automorphisms which might not be fully irreducible but exhibit some properties similar to fully-irreducible elements.

This work is an extension of a construction in \cite{self1}. In that paper the author constructs free subgroups in $\out$ 
where they start with exponentially growing elements $\phi, \psi \in \out$ (not necessarily fully-irreducible) and find sufficient conditions 
so that the elements of the free group, not powers of or conjugate to some power of $\phi, \psi$, are hyperbolic and fully-irreducible 
(recall that it was shown by Bestvina-Feighn\cite{BF-92} and Brinkmann \cite{Br-00} that an outer automorphism was \textit{hyperbolic} if and only if it did not have any periodic conjugacy classes). 
One of the key assumptions in proving this was that the elements $\phi, \psi$ did not 
have a common periodic free factor system or conjugacy class. In this paper we deal with the case when $\phi, \psi$ do have a nontrivial common invariant free factor system.

Given a collection of free factors $F^1, F^2,...., F^k$ of $\F$, such that $\F= F^1\ast F^2\ast ... \ast F^k \ast B$ with $B$ possibly trivial, 
we say that the collection forms a \emph{free factor system}, written as  $\mathcal{F}:=\{[F^1], [F^2],...., [F^k]\}$ and we say that 
a conjugacy class $[c]$ of a word $c\in\F$ is carried by $\mathcal{F}$ if there exists some $1\leq i \leq k$ and a representative $H^i$ of $F^i$ 
such that $c\in H^i$.

The notion of $\phi$ being fully irreducible relative to a free factor system $\mathcal{F}$ intuitively can be thought of as $\phi$
being fully irreducible in the ``complement'' of $\mathcal{F}$ (see beginning of Section \ref{3} for the definition). 
Such examples exist in abundance, especially when rank($\F$) is high, and are very easy to construct 
by ``gluing'' a fully irreducible outer automorphism of some free factor $K$ of $\F$ together with another outer automorphism 
defined on the complementary free factor $K'$ of $\F$, where $K\ast K' =\F$. The resulting automorphism will be fully irreducible relative to the free factor system 
$\{[K']\}$.

We now state our main construction in this paper:
\begin{thma}
 
Given a free factor system $\mathcal{F}$ with co-edge number $\geq 2$, given $\phi, \psi \in \out$ each preserving $\mathcal{F}$, and given invariant lamination pairs 
  $\Lambda^\pm_\phi, \Lambda^\pm_\psi$, so that the pair $(\phi, \Lambda^\pm_\phi), (\psi, \Lambda^\pm_\psi)$ is independent relative to $\mathcal{F}$, then there $\exists$  $M\geq 1$, 
  such that for any integer $m,n \geq M$, the group $\langle \phi^m, \psi^n \rangle$ is a free group of rank 2, all of whose non-trivial elements except perhaps the powers of $\phi, \psi$ 
  and their conjugates, are fully irreducible relative to $\mathcal{F}$ with a lamination pair which fills relative to $\mathcal{F}$. 
  
  In addition if both $\Lambda^\pm_\phi, \Lambda^\pm_\psi$ are non-geometric 
  then this lamination pair is also non-geometric.
\end{thma}
  
The assumption of ``co-edge number $\geq 2$'' is a technical condition helps us get rid of some pathological cases that may arise. To understand such
situations, recall that from the work of 
Feighn-Handel in their ``Recognition theorem'' work \cite{FH-11}, we may pass to a rotationless power of $\phi$ and choose a completely split 
train track map (\emph{CT}) $f: G\to G$, that has a filtration element $G_{r-1}$ which realizes $\mathcal{F}$. The ``co-edge number $\geq 2$'' ensures 
that the strata $H_r$ is an exponentially growing strata and not a polynomially growing one and hence guarantees the existence of an attracting lamination 
associated to $H_r$. Roughly speaking, this co-edge assumption ensures that conjugacy classes that are not carried by $\mathcal{F}$ grow exponentially under iteration of 
$\phi$ (see Lemma \ref{firelf}).

The notion of independence of the pseudo-Anosov elements in $\mathcal{MCG}(\mathcal{S})$ is equivalent to the property that the attracting and repelling laminations of the two elements are 
mutually transverse on the surface, 
from which it follows that the collection of laminations fills (in fact they individually fill). 
These filling properties are enjoyed by fully irreducible elements in $\out$. 
But exponentially growing elements which are not fully irreducible might not have this property. In the definition of \textquotedblleft pairwise independence rel $\mathcal{F}$ \textquotedblright \ref{relind} 
we extract a list the properties similar to pseudo-Anosov maps that make the aforementioned theorem work.

\textbf{Idea of Proof : } 
The tools we use to prove our theorem are the theory of relative train track maps and the weak attraction property, which were developed in a series of works in 
\cite{BH-92}, \cite{BFH-97}, \cite{BFH-00}, \cite{FH-11}, \cite{HM-09} and \cite{HM-13c}. Given exponentially growing elements of $\out$, 
which are pairwise independent relative (see \ref{relind}) to a free factor system $\mathcal{F}$, 
we generalize a pingpong type argument developed by Handel-Mosher in \cite{HM-13c} to produce exponentially growing elements. This part of the proof 
uses the mutual attraction property in the pairwise independent hypothesis to bounce legal long segments of generic leaves of the given attracting 
laminations and make them grow exponentially.
Then we proceed to show that these elements will be fully irreducible relative to $\mathcal{F}$ by using Stallings graphs, which again was 
originally developed in \cite{HM-13c}.

The aforementioned theorem was originally suggested by Lee Mosher to the author and written in relation to the 
second-bounded cohomology alternative paper \cite{HM-15} by Handel and Mosher. They show that every subgroup of $\out$ is either virtually abelian or has uncountably infinite second bounded cohomological dimension.
They use our main theorem and it's corollary \ref{rfi} as a special type of relatively irreducible free subgroups
and develop a method to reduce their general case to the special case.

In the later half of the paper, Section \ref{section4}, we proceed to show that under some natural conditions
 subgroups constructed above will yield strongly relative hyperbolic extensions and this shows the rich geometric properties that relatively fully irreducible 
outer automorphisms have. The first step of the proof is to show that outer automorphisms which are fully irreducible relative to a free factor system $\mathcal{F}$ give 
(strongly) relatively hyperbolic extension groups in the form of the following theorem:

\newtheorem*{thmb}{Theorem \ref{relhyp2}}
 \begin{thmb}
  Let $\phi\in\out$  be rotationless and $\mathcal{F}=\{[F^1], [F^2],..., [F^k]\}$ be a $\phi-$invariant free factor system such that 
  $\mathcal{F}\sqsubset\{[\F]\}$  
  is a multi-edge extension and $\phi$ is fully irreducible relative to $\mathcal{F}$ and nongeometric above $\mathcal{F}$. 
  Then the extension group $\Gamma$ in the short exact sequence
  $$1 \to \F \to \Gamma \to \langle \phi \rangle \to 1$$ is strongly hyperbolic relative
  to the collection of subgroups $\{F^i{\rtimes_{\Phi_i}} \mathbb{Z}\}$, where $\Phi_i$ is a 
  chosen lift of $\phi$ such that $\Phi_i(F^i)=F^i$.
 \end{thmb}
 
Intuitively, this theorem is a consequence of the observation that we made earlier, namely every conjugacy class not carried by $\mathcal{F}$ grows 
exponentially under iteration of $\phi$. Also notice that if we take $\mathcal{F}=\emptyset$, this theorem gives another proof of 
the result of Brinkmann and Bestvina-Feighn which shows that the extension group 
$\Gamma$ is hyperbolic if $\phi$ is does not have periodic conjugacy classes.

\textbf{Idea of proof : } Using Proposition \cite[Proposition 2.2]{HM-13d} (stated here as Lemma \ref{firelf}) we can conclude that there 
exists a dual lamination pair $\Lambda^\pm_\phi$ such that the nonattracting subgroup system $\mathcal{A}_{na}(\Lambda^\pm_\phi)=\mathcal{F}$. 
Recall that a conjugacy class is not weakly attracted to $\Lambda^+_\phi$ if and only if it is carried by the nonattracting subgroup system. 
Hence we can safely conclude that every conjugacy class not carried by $\mathcal{F}$ is exponentially growing. Next we generalize the notion of 
\emph{legality} of circuits (see definition \ref{leg}) which first appeared in \cite{BFH-97} and use the weak attraction theorem to show that every 
conjugacy class not carried by $\mathcal{F}$, when iterated by either $\phi$ or $\phi^{-1}$ at most $M_2$ times (Lemma \ref{legality}), 
gains sufficiently long legal segments. This result is then used to prove ``conjugacy flaring'' in Lemma \ref{conjflare} and ``strictly flaring'' 
condition in Proposition \ref{strictflare}. The technique of proof in both these flaring results is a generalization of the technique used in 
\cite{BFH-97}. The last part of the proof is to verify that we satisfy all the conditions of the Mj-Reeves strong combination theorem
for relatively hyperbolic groups \cite[Theorem 4.6]{MjR-08} and we conclude relative hyperbolicity using their combination theorem.

The key idea that allows us to use the Mj-Reeves combination theorem is that the nonattracting subgroup system  
 $\mathcal{A}_{na}(\Lambda^\pm_\phi)= \{[F^1], [F^2],...., [F^k]\}$ is a \emph{malnormal} subgroup system (i.e. $H^i\cap H^j=\emptyset$ for any 
 representatives of $F^i, F^j$ where $i\neq j$). This result is due to Handel and Mosher \cite{HM-13c} (see \ref{NAS}) and allows us to 
 choose a collection of representatives $\{F^i\}$ and perform an electrocution of $\F$ with respect to this collection and the consequence is that 
 $\F$ is (strongly) hyperbolic relative to the collection of subgroups $\{F^i\}$ (since free factors are quasiconvex). This follows from 
 a standard fact in the theory of relatively hyperbolic groups which states that any hyperbolic group is strongly hyperbolic 
 relative to a malnormal collection of quasiconvex subgroups.
 
 It is worth pointing out, that this is the first such use of the nonattracting subgroup system in the available literature on this topic and 
 establishes the connection between the peripheral subgroups used for showing relative hyperbolicity and the nonattracting subgroup system.
 We sincerely believe that this connection can be used for proving further interesting results in the area.

Notice that although the hypothesis here includes the ``nongeometric'' assumption, we have also 
proven the case for geometric extensions in \ref{geoext} and generalizes a well known result for surface group with punctures, which shows 
that the mapping tori of a pseudo-Anosov mapping class on a surface with punctures is strongly hyperbolic relative to the collection of 
cusp groups.

Then we proceed to prove the main theorem of this work,  where we produce free-by-free (strongly) relatively hyperbolic groups by 
using the above theorem. The method of proof here is to again apply 
the Mj-Reeves strong combination theorem for relatively hyperbolic groups. This is done in Proposition \ref{34} by proving 
a version of the 3-of-4 stretch lemma due to Lee Mosher and we use it to show that all conditions of the Mj-Reeves combination theorem is 
satisfied and hence we have:

\newtheorem*{thmc}{Theorem \ref{nongeoext}}
 \begin{thmc}
  Suppose $\phi,\psi\in\out$ are rotationless and 
  $\mathcal{F}=\{[F^1], [F^2],..., [F^k]\}$ be a $\phi, \psi-$invariant free factor system such that 
  $\mathcal{F}\sqsubset\{[\F]\}$  
  is a multi-edge extension and $\phi, \psi$ are fully irreducible relative to $\mathcal{F}$, pairwise 
  independent relative to $\mathcal{F}$
  and both are nongeometric above $\mathcal{F}$.
  If $Q=\langle \phi^m, \psi^n \rangle$ denotes 
  the free group in the conclusion of corollary \ref{rfi}, then
  the extension group $\Gamma$ in the short exact sequence 
  $$ 1 \to \F \to \Gamma \to Q \to 1 $$ is strongly relatively hyperbolic with respect to the 
  collection of subgroups 
  $\{F^i \rtimes \widehat{Q_i} \}$, where  $\widehat{Q_i}$  is a lift that preserves $F^i$
  
 \end{thmc}
 The geometric version of this theorem follows in \ref{geoext}. In the setting of surface group 
 with punctures, the result was proven by Mj-Reeves in \cite[Theorem 4.9]{MjR-08}.

As far as the knowledge of the author goes, this is the 
 first example of free-by-free (strongly) relatively hyperbolic groups in the study of $\out$, which lies outside the scope of mapping class groups 
 of surfaces. Examples of free-by-free hyperbolic groups have been constructed on multiple occasions by 
 Bestvina-Feighn-Handel \cite{BFH-97}, Kapovich-Lustig \cite{KL-10} and in both these constructions the elements of the quotient group 
 are fully irreducible and hyperbolic. The constructions of Bestvina-Feighn-Handel and Kapovich-Lustig follow as a corollary of the above theorem by taking 
 $\mathcal{F}=\emptyset$. 
 Dowdall-Taylor \cite{DT-18} constructed examples of free-by-(convex cocompact) examples of hyperbolic groups and their constructions are so far the most general 
 in the cases where the conditions imply that every element of the quotient group is either fully irreducible and hyperbolic or has finite order.
 A different class of examples of free-by-free hyperbolic groups which do not assume that every element of the 
 quotient group is fully irreducible has been due to Uyanik \cite{Uya-17} and by the author's follow up work \cite{Gh-18} which uses techniques developed 
 in this paper.

\section{Preliminaries}

\subsection{Marked graphs, circuits and path: } A \textit{marked graph} is a graph $G$ which has no valence 1 vertices and equipped with 
 a homotopy equivalence to the rose $m: G\to R_n$ (where $n = \text{rank}(\F)$). The fundamental group of $G$ therefore can be identified with 
 $\F$ up to inner automorphism. A \textit{circuit} in a marked graph is an immersion (i.e. locally injective and continuous map)
 of $S^1$ into $G$. The set of circuits in $G$ can be identified 
 with the set of conjugacy classes in $\F$. Similarly a path is an immersion of the interval $[0, 1]$ into $G$. Given any continuous map 
 from $S^1$ or $[0,1]$ it can be \textit{tightened} to a circuit or path, meaning the original map is freely homotopic to a locally injective and continuous 
 map from the respective domains. In fact, given any homotopically nontrivial and continuous map from $S^1$ to $G$, it can be tightened to a 
  unique circuit. We shall not distinguish between circuits or paths that differ by a homeomorphism of their respective domains.

\subsection{EG strata, NEG strata and Zero strata:}
  
 A \textit{filtration} of a marked graph $G$ is a strictly increasing sequence of subgraphs 
 $G_0 \subset G_1 \subset \cdots \subset G_k = G$, each with no isolated vertices. 
 The individual terms $G_k$ are called \textit{filtration elements}, and if $G_k$ is a core graph (i.e. a graph without valence 1 vertices)
 then it is called a 
 \textit{core filtration element}. The subgraph $H_k = G_k \setminus G_{k-1}$ together with the vertices which occur as endpoints of edges in 
 $H_k$ is called the \textit{stratum of height $k$}.
 The \textit{height}\index{height} of subset of $G$ is the minimum $k$ such that the subset is contained in $G_k$. 
 The height of a map to $G$ is the height of the image of the map. 
 A \textit{connecting path} of a stratum $H_k$ is a nontrivial finite path $\gamma$ of height $< k$ whose endpoints are contained in 
 $H_k$.

Given a topological representative $f : G \to G$ of $\phi \in \out$, we say that $f$ \textit{respects} the filtration or 
that the filtration is \textit{$f$-invariant} if $f(G_k) \subset G_k$ for all $k$. If this is the case then we also say that the 
filtration is \textit{reduced} if for each free factor system $\mathcal{F}$ which is invariant under $\phi^i$ for some $i \ge 1$, if 
$[\pi_1 G_{r-1}] \sqsubset \mathcal{F} \sqsubset [\pi_1 G_r]$ then either $\mathcal{F} = [\pi_1 G_{r-1}]$ or $\mathcal{F} = [\pi_1 G_r]$.

Given an $f$-invariant filtration, for each stratum $H_k$ with edges $\{E_1,\ldots,E_m\}$, define
the \textit{transition matrix} of $H_k$ to be the square matrix whose $j^{\text{th}}$ column records the number of times 
$f(E_j)$ crosses the other edges. If $M_k$ is the zero matrix then we say that $H_k$ is a \textit{zero stratum}. 
If $M_k$ irreducible --- meaning that for each $i,j$ there exists $p$ such that the $i,j$ entry of the $p^{\text{th}}$ 
power of the matrix is nonzero --- then we say that $H_k$ is \textit{irreducible}; and if one can furthermore choose $p$ independently 
of $i,j$ then $H_k$ is \textit{aperiodic}. Assuming that $H_k$ is irreducible, by  Perron-Frobenius theorem, the matrix $M_k$ 
a unique eigenvalue $\lambda \ge 1$, called the \textit{Perron-Frobenius eigenvalue}, for which some associated eigenvector has 
positive entries: if $\lambda>1$ then we say that $H_k$ is an \textit{exponentially growing} or EG stratum; whereas if $\lambda=1$ 
then $H_k$ is a \textit{nonexponentially growing} or NEG stratum. An NEG stratum that is neither fixed nor linear is called a \emph{superlinear} edge.

\subsection{Weak topology}
\label{sec:2}
Given any finite graph $G$, let $\widehat{\mathcal{B}}(G)$ denote the compact space of equivalence classes of circuits in $G$ and paths in $G$, whose endpoints (if any) are vertices of $G$. We give this space the \textit{weak topology}.
Namely, for each finite path $\gamma$ in $G$, we have one basis element $\widehat{N}(G,\gamma)$ which contains all paths and circuits in $\widehat{\mathcal{B}}(G)$ which have $\gamma$ as its subpath.
Let $\mathcal{B}(G)\subset \widehat{\mathcal{B}}(G)$ be the compact subspace of all lines in $G$ with the induced topology. One can give an equivalent description of $\mathcal{B}(G)$ following \cite{BFH-00}.
A line is completely determined, up to reversal of direction, by two distinct points in $\partial \mathbb{F}$, since there only one line that joins these two points. 
We can then induce the weak topology on the set of lines coming from the Cantor set $\partial \mathbb{F}$. More explicitly,
let $\widetilde{\mathcal{B}}=\{ \partial \mathbb{F} \times \partial \mathbb{F} - \vartriangle \}/(\mathbb{Z}_2)$, where $\vartriangle$ is the diagonal and $\mathbb{Z}_2$ acts by interchanging factors. We can put the weak topology on
$\widetilde{\mathcal{B}}$, induced by Cantor topology on $\partial \mathbb{F}$. The group $\mathbb{F}$ acts on $\widetilde{\mathcal{B}}$ with a compact but non-Hausdorff quotient space $\mathcal{B}=\widetilde{\mathcal{B}}/\mathbb{F}$.
The quotient topology is also called the \textit{weak topology}. Elements of $\mathcal{B}$ are called \textit{lines}. A lift of a line $\gamma \in \mathcal{B}$ is an element  $\widetilde{\gamma}\in \widetilde{\mathcal{B}}$ that
projects to $\gamma$ under the quotient map and the two elements of $\widetilde{\gamma}$ are called its endpoints.

One can naturally identify the two spaces $\mathcal{B}(G)$ and $\mathcal{B}$ by considering a homeomorphism between the two Cantor sets $\partial \mathbb{F}$ and set of ends of universal cover of $G$ , where $G$ is a marked graph.
$\out$ has a natural action on  $\mathcal{B}$. The action comes from the action of Aut($\mathbb{F}$) on $\partial \mathbb{F}$. Given any two marked graphs $G,G'$ and a homotopy equivalence $f:G\rightarrow G'$ between them, the induced map
$f_\#: \widehat{\mathcal{B}}(G)\rightarrow \widehat{\mathcal{B}}(G')$ is continuous and the restriction $f_\#:\mathcal{B}(G)\rightarrow \mathcal{B}(G')$ is a homeomorphism. With respect to the identification
$\mathcal{B}(G)\approx \mathcal{B}\approx \mathcal{B}(G')$, if $f$ preserves the marking then $f_{\#}:\mathcal{B}(G)\rightarrow \mathcal{B}(G')$ is equal to the identity map on $\mathcal{B}$. When $G=G'$, $f_{\#}$ agree with their
homeomorphism $\mathcal{B}\rightarrow \mathcal{B}$ induced by the outer automorphism associated to $f$.

Given a marked graph $G$, a \textit{ray} in $G$ is an one-sided infinite concatenation of edges $E_0E_1E_2......$. A \textit{ray} of $\mathbb{F}$ is an
element of the orbit set $\partial\mathbb{F}/\mathbb{F}$. There is connection between these two objects which can be explained as follows.
Two  rays in $G$ are asymptotic if they have equal subrays, and
this is an equivalence relation on the set of rays in $G$. The set of asymptotic equivalence
classes of rays $\rho$ in $G$ is in natural bijection with $\partial\mathbb{F}/\mathbb{F}$ where $\rho$ in $G$ corresponds to end
$\xi\in\partial\mathbb{F}/\mathbb{F}$ if there is a lift $\tilde{\rho} \subset G$ of $\rho$ and a lift
$\tilde{\xi}\in\partial\mathbb{F}$ of $\xi$, such that $\tilde{\rho}$ converges to $\tilde{\xi}$ in the
Gromov compactification of $\tilde{G}$. A ray $\rho$ is often said to be the realization of $\xi$ if the above
conditions are satisfied.

A line(path) $\gamma$ is said to be \textit{weakly attracted} to a line(path) $\beta$ under the action of $\phi\in\out$, if the $\phi^k(\gamma)$ converges to $\beta$ in the weak topology. This is same as saying, for any given finite subpath of $\beta$, $\phi^k(\gamma)$
contains that subpath for some value of $k$; similarly if we have a homotopy equivalence $f:G\rightarrow G$,  a line(path) $\gamma$ is said to be \textit{weakly attracted} to a line(path) $\beta$ under the action of $f_{\#}$ if the $f_{\#}^k(\gamma)$ converges to $\beta$
in the weak topology. The \textit{accumulation set} of a ray $\gamma$ in $G$ is the set of lines  $l\in \mathcal{B}(G)$ which are elements of the weak closure of $\gamma$; which is same as saying every finite subpath of $l$
occurs infinitely many times as a subpath $\gamma$. The weak accumulation set of some $\xi\in\partial\mathbb{F}$ is the set of lines in the weak closure
of any of the asymptotic rays in its equivalence class. We call this the \textit{weak closure} of $\xi$.

\subsection{Free factor systems and subgroup systems}
\label{sec:3}

Define a \textit{subgroup system} $\mathcal{A} = \{[H_1], [H_2], .... ,[H_k]\}$ to be a finite collection of conjugacy classes of finite rank subgroups $H_i<\F$. Define a subgroup system to be \textit{malnormal} if for any $[H_i],[H_j]\in \mathcal{A}$, 
if $H_i^x\cap H_j$ is nontrivial then $i=j$ and $x\in H_i$. Two subgroup systems $\mathcal{A}$ and $\mathcal{A}'$ are said to be \textit{\textbf{mutually malnormal}} if both $H_i^x\cap H'_j$ and $H_i\cap (H'_j)^x$ are trivial for every 
$[H_i]\in \mathcal{A}, [H'_j]\in \mathcal{A}'$ and $x\in \F$.  Given two subgroup systems $\mathcal{A}, \mathcal{A}'$ we give a partial ordering to set of all free factor systems by defining 
$\mathcal{A}\sqsubset \mathcal{A}'$ if for each conjugacy class of subgroup $[A]\in \mathcal{A}$ 
there exists some conjugacy class of subgroup $[A']\in \mathcal{A}'$ such that $A < A'$.

Given a finite collection $\{K_1, K_2,.....,K_s\}$ of subgroups of $\F$ , we say that this collection determines a \textit{free factorization} of $\F$ if $\F$ is the free product of these subgroups, that is, 
$\F = K_1 * K_2 * .....* K_s$. The conjugacy class of a subgroup is denoted by [$K_i$]. 
A \emph{free factor system} is a finite collection of conjugacy classes of subgroups of $\F$ , $\mathcal{K}:=\{[K_1], [K_2],.... [K_p]\}$ such that there is a free factorization of $\F$ of the form 
$\F = K_1 * K_2 * ....* B$, where $B$ is some finite rank subgroup of $\F$ (it may be trivial). 
There is an action of $\out$ on the set of all conjugacy classes of subgroups of $\F$. This action induces an action of $\out$ on the set of all free factor systems. For notation simplicity we will avoid writing $[K]$ all the time and write $K$ instead, when we discuss   
the action of $\out$ on this conjugacy class of subgroup $K$ or anything regarding the conjugacy class [$K$]. It will be understood that we actually mean [$K$].

For any marked graph $G$ and any subgraph $H \subset G$, the fundamental groups of the noncontractible components of $H$ form a free factor system . We denote this by $[H]$. A subgraph of $G$ which has no valence 1 vertex is called a \textit{core graph}. 
Every subgraph has a unique core graph, which is a deformation retract of its noncontractible components.
A free factor system $\mathcal{K}$ \textit{carries a conjugacy class} $[c]$ in $\F$ if there exists some $[K] \in \mathcal{K}$ such that $c\in K$. We say that $\mathcal{K}$ \textit{carries the line} $\gamma \in \mathcal{B}$ if for any marked graph $G$ 
the realization of $\gamma$ in $G$ is the weak limit of a sequence of circuits in $G$ each of which is carried by $\mathcal{K}$.
An equivalent way of saying this is: for any marked graph $G$ and a subgraph $H \subset G $ with $[H]=\mathcal{K}$, the realization of $\gamma$ in $G$ is contained in $H$.

A subgroup system $\mathcal{A}$ carries a conjugacy class $[c]\in \F$ if there exists some $[A]\in\mathcal{A}$ such that $c\in A$. Also, we say that $\mathcal{A}$ carries a line $\gamma$ if one of the following equivalent conditions hold:
\begin{itemize}
 \item $\gamma$ is the weak limit of a sequence of conjugacy classes carries by $\mathcal{A}$.
 \item There exists some $[A]\in \mathcal{A}$ and a lift $\widetilde{\gamma}$ of $\gamma$ so that the endpoints of $\widetilde{\gamma}$ are in $\partial A$.
 
\end{itemize}
The following fact is an important property of lines carried by a subgroup system. The proof is by using the observation that $A<\F$ is of finite rank implies that $\partial A$ is a compact subset of $\partial \F$
\begin{lemma}
 For each subgroup system $\mathcal{A}$ the set of lines carried by $\mathcal{A}$ is a closed subset of $\mathcal{B}$
\end{lemma}

Given a subgroup systems $\mathcal{A}, \mathcal{A}'$ and a free factor system $\mathcal{F}$ such that $\mathcal{F}\sqsubset\mathcal{A}$ and 
$\mathcal{F}\sqsubset\mathcal{A}'$ we say that the subgroup systems $\mathcal{A}, \mathcal{A}'$ are \textbf{\emph{mutually malnormal relative to $\mathcal{F}$}} 
when a conjugacy class $[c]$ is carried by both $\mathcal{A}$ and $\mathcal{A}'$ if and only if $[c]$ is carried by $\mathcal{F}$.

The following lemma describes the \emph{meet} of free factor systems and is used to define the free factor support of a set of lines.
\begin{lemma}
 [\cite{BFH-00}, Section 2.6] Every collection $\{\mathcal{K}_i\}$ of free factor systems has a well-defined meet $\wedge\{\mathcal{K}_i\} $, which is the unique maximal free factor system $\mathcal{K}$ such that $\mathcal{K}\sqsubset \mathcal{K}_i$ for all $i$. Moreover, 
 for any free factor $F< \F$ we have $[F]\in \wedge\{\mathcal{K}_i\}$ if and only if there exists an indexed collection of subgroups $\{F_i\}_{i\in I}$ such that $[A_i]\in \mathcal{K}_i$ for each $i$ and $A=\bigcap_{i\in I} A_i$. 
\end{lemma}

From \cite{BFH-00}
The \textit{free factor support} of a set of lines $B$ in $\mathcal{B}$ is (denoted by $\mathcal{F}_{supp}(B)$) defined as the meet of all free factor systems that carries $B$. If $B$ is a single line then $\mathcal{F}_{supp}(B)$ is 
single free factor. We say that a set of lines, $B$, is \textit{filling} if $\mathcal{F}_{supp}(B)=[\F]$

\subsection{Topological representatives and Train track maps}
Given $\phi\in\out$ a \textit{topological representative} is a homotopy equivalence $f:G\rightarrow G$ such that $\rho: R_r \rightarrow G$ is a marked graph, 
$f$ takes vertices to vertices and edges to paths and $\overline{\rho}\circ f \circ \rho: R_r \rightarrow R_r$ represents $R_r$. A nontrivial path $\gamma$ in $G$ 
is a \textit{periodic Nielsen path} if there exists a $k$ such that $f^k_\#(\gamma)=\gamma$;
the minimal such $k$ is called the period and if $k=1$, we call such a path \textit{Nielsen path}. A periodic Nielsen path is \textit{indivisible} if it cannot 
be written as a concatenation of two or more nontrivial periodic Nielsen paths.

Given a subgraph $H\subset G$ let $G\setminus H$ denote the union of edges in $G$ that are not in $H$.

Given a marked graph $G$ and a homotopy equivalence $f:G\rightarrow G$ that takes edges to paths, one can define a new map $Tf$ by setting $Tf(E)$ 
to be the first edge in the edge path associated to $f(E)$; similarly let $Tf(E_i,E_j) = (Tf(E_i),Tf(E_j))$. So $Tf$ is a map that takes turns to turns. We say that a 
non-degenerate turn is illegal if for some iterate of $Tf$ the turn becomes degenerate; otherwise the
 turn is legal. A path is said to be legal if it contains only legal turns and it is $r-legal$ if it is of height $r$ and all its illegal turns are in $G_{r-1}$.

 \textbf{Relative train track map.} Given $\phi\in \out$ and a topological representative $f:G\rightarrow G$ with a filtration $G_0\subset G_1\subset \cdot\cdot\cdot\subset G_k$ which is preserved by $f$, 
 we say that $f$ is a train relative train track map if the following conditions are satisfied: \label{rtt}
 \begin{enumerate}
  \item $f$ maps r-legal paths to legal r-paths.
  \item If $\gamma$ is a connecting path of $H_r$, then $f_\#(\gamma)$ is a connecting path of $H_r$.
  \item If $E$ is an edge in $H_r$ then $Tf(E)$ is an edge in $H_r$. $Df$ maps the set of directions of height $r$ to itself. In particular, 
  every turn consisting of a direction of height $r$ and one of height $< r$  is legal.
 \end{enumerate}

 The following result is in \cite[Theorem 5.1.5]{BFH-00} which assures the existence of a good kind of relative train track map which is the one 
 we will be needing here.
 
 \begin{lemma}
  Suppose $\phi\in \out$ and $\mathcal{F}$ is a $\phi-$invariant free factor system. Then $\phi$ 
  has a improved relative train track representative with a filtration element that realizes $\mathcal{F}$.
 \end{lemma}

\subsection{Attracting Laminations and their properties under CTs}
\label{sec:6}
For any marked graph $G$, the natural identification $\mathcal{B}\approx \mathcal{B}(G)$ induces a bijection between the closed subsets of $\mathcal{B}$ and the closed subsets of $\mathcal{B}(G)$. A closed 
subset in any of these two cases is called a \textit{lamination}, denoted by $\Lambda$. Given a lamination $\Lambda\subset \mathcal{B}$ we look at the corresponding lamination in $\mathcal{B}(G)$ as the 
realization of $\Lambda$ in $G$. An element $\lambda\in \Lambda$ is called a \textit{leaf} of the lamination.\\
A lamination $\Lambda$ is called an \textit{attracting lamination} for $\phi$ is it is the weak closure of a line $l$ (called the \textit{generic leaf of $\lambda$}) satisfying the following conditions:
\begin{itemize}
 \item $l$ is bi-recurrent leaf of $\Lambda$.
\item $l$ has an \textit{attracting neighborhood} $V$, in the weak topology, with the property that every line in $V$ is weakly attracted to $l$.
\item no lift $\widetilde{l}\in \mathcal{B}$ of $l$ is the axis of a generator of a rank 1 free factor of $\F$ .
\end{itemize}

We know from \cite{BFH-00} that with each $\phi\in \out$ we have a finite set of laminations $\mathcal{L}(\phi)$, called the set of \textit{attracting laminations} of $\phi$, and the set $\mathcal{L}(\phi)$ is 
invariant under the action of $\phi$. When it is nonempty $\phi$ can permute the elements of $\mathcal{L}(\phi)$ if $\phi$ is not rotationless. For rotationless $\phi$ $\mathcal{L}(\phi)$ is a fixed set. Attracting laminations are directly related to EG stratas. The following fact is a result from \cite{BFH-00} section 3.

\textbf{Dual lamination pairs. }
We have already seen that the set of lines carried by a free factor system is a closed set and so, together with the fact that the weak closure of a generic leaf $\lambda$ of an attracting lamination $\Lambda$ is the whole lamination $\Lambda$ tells us that 
$\mathcal{F}_{supp}(\lambda) = \mathcal{F}_{supp}(\Lambda)$. In particular the free factor support of an attracting lamination $\Lambda$ is a single free factor.
Let $\phi\in \out$ be an outer automorphism and $\Lambda^+_\phi$ be an attracting lamination of $\phi$ and $\Lambda^-_\phi$ be an attracting lamination of $\phi^{-1}$. We say that this lamination pair is a \textit{dual lamination pair} if $\mathcal{F}_{supp}(\Lambda^+_\phi) = \mathcal{F}_{supp}(\Lambda^-_\phi)$. 
By Lemma 3.2.4 of \cite{BFH-00} there is bijection between $\mathcal{L}(\phi)$ and $\mathcal{L}(\phi^{-1})$ induced by this duality relation. 
The following fact is Lemma 2.35 in \cite{HM-13c}; it establishes an important property of lamination pairs in terms of inclusion. We will use it in proving duality for the attracting and repelling laminations we produce in Proposition \ref{pingpong}. 
\begin{lemma}
\label{lam_incl}
 If $\Lambda^\pm_i, \Lambda^\pm_j$ are two dual lamination pairs for $\phi\in \out$  then $\Lambda^+_i\subset \Lambda^+_j$ if and only if $\Lambda^-_i\subset \Lambda^-_j$.
\end{lemma}

\subsection{Nonattracting subgroup system $\mathcal{A}_{na}(\Lambda^+_\phi)$}
\label{sec:7}
The \textit{nonattracting subgroup system} of an attracting lamination contains information about lines and circuits which are not attracted to the lamination.
This is one the crucial ingredients of our proof here. First introduced by Bestvina-Feighn-Handel in \cite{BFH-00}, this was later explored in great details 
by Handel-Mosher in \cite{HM-13c}. The first step to describing the nonattracting subgroup system is to 
construct the \emph{nonattracting subgraph} as follows:

\begin{definition}
 \emph{Suppose $\phi\in\out$ is exponentially growing and $f:G\rightarrow G$ is an improved relative train track map
 representing $\phi$ such that $\Lambda^+_\phi$ is an invariant attracting lamination which corresponds to the EG stratum $H_s\in G$.
 The \textbf{nonattracting subgraph $\mathcal{Z}$} of $G$ is defined as a union of irreducible stratas $H_i$ of $G$ such that no edge in 
 $H_i$ is weakly attracted to $\Lambda^+_\phi$. This is equivalent to saying that a strata $H_r\subset G\setminus \mathcal{Z}$ if and only if there exists an 
 edge $E_r$ in $H_r$ and some $k\geq 0$ such that 
 some term in the complete splitting of $f^k_\#(E_r)$ is an edge in $H_s$. }
\end{definition}

Define the path $\widehat{\rho}_s$ to be trivial path at any chosen vertex if there does not exist any indivisible Nielsen path of height $s$, 
otherwise $\widehat{\rho}_s$ is the unique  indivisible path of height $s$.

\textbf{The groupoid  $\langle \mathcal{Z}, \widehat{\rho}_s \rangle$ - } Let $\langle \mathcal{Z}, \widehat{\rho}_s \rangle$ be the set of lines, rays, circuits and finite paths in $G$ which can be written 
as a concatenation of 
subpaths, each of which is an edge in $\mathcal{Z}$, the path $\widehat{\rho}_s$ or its inverse. Under the operation of tightened concatenation of paths
in $G$,
this set forms a groupoid (Lemma 5.6, [\cite{HM-13c}]).


Define the graph $K$ by setting $K=\mathcal{Z}$ if $\widehat{\rho}_s$ is trivial and let $h:K \rightarrow G$ be the inclusion map. Otherwise define an edge $E_\rho$ representing the domain of the Nielsen path $\rho_s:E_\rho \rightarrow G_s$, and let $K$ be the disjoint union of $\mathcal{Z}$ and $E_\rho$ with the following identification.
 Given an endpoint $x\in E_\rho$, if $\rho_s(x)\in \mathcal{Z}$ then identify $x\sim\rho_s(x)$.Given distinct endpoints $x,y\in E_\rho$, if $\rho_s(x)=\rho_s(y)\notin \mathcal{Z}$ then identify $x \sim y$. In this case define $h:K\rightarrow G$ to be the inclusion map on $K$ and the map $\rho_s$ on $E_\rho$. It is not difficult to see that the map $h$ is an immersion.
 Hence restricting $h$ to each component of $K$, we get an injection at the level of fundamental groups. The 
 \textbf{\emph{nonattracting subgroup system}} $\mathcal{A}_{na}(\Lambda^+_\phi)$ is defined to be the subgroup system defined by this immersion.

We will leave it to the reader to look it up in \cite{HM-13c} where it is explored in details. We however list some key properties which we will be using and justifies the importance of this
subgroup system.
\begin{lemma}(\cite{HM-13c}- Lemma 1.5, 1.6)
\label{NAS}
 \begin{enumerate}
  \item The set of lines carried by $\mathcal{A}_{na}(\Lambda^+_\phi)$ is closed in the weak topology.
  \item A conjugacy class $[c]$ is not attracted to $\Lambda^+_\phi$ if and only if it is carried by $\mathcal{A}_{na}(\Lambda^+_\phi)$.
  \item $\mathcal{A}_{na}(\Lambda^+_\phi)$ does not depend on the choice of the CT representing $\phi$.
  \item  Given $\phi, \phi^{-1} \in \out$ both rotationless elements and a dual lamination pair $\Lambda^\pm_\phi$ we have $\mathcal{A}_{na}(\Lambda^+_\phi)= \mathcal{A}_{na}(\Lambda^-_\phi)$
  \item $\mathcal{A}_{na}(\Lambda^+_\phi)$ is a free factor system if and only if the stratum $H_r$ is not geometric.
  \item $\mathcal{A}_{na}(\Lambda^+_\phi)$ is malnormal.

 \end{enumerate}

\end{lemma}

\begin{remark}
 \emph{It is a part of the definition of the nonattracting subgroup system that $\Lambda^+_\phi$ has to be invariant under $\phi$. However, this might not always be the case. 
 For instance, we could have that $\phi$ permutes the elements of $\mathcal{L}(\phi)$ (the set of all attracting laminations for $\phi$). In this case one needs to pass to 
 some finite power of $\phi$ to stabilize $\Lambda^+_\phi$ and then define the nonattracting subgroup system. Feighn-Handel's work ensures that there is a maximal upper 
 bound to the power to which $\phi$ needs to be raised to achieve this, namely the \emph{rotationless power} \cite{FH-11}}.
\end{remark}

\subsection{Weak attraction theorem }
\label{sec:9}
\begin{lemma}[\cite{HM-13c} Corollary 2.17]
\label{WAT}
 Let $\phi\in \out$ be a rotationless and exponentially growing. Let $\Lambda^\pm_\phi$ be a dual lamination pair for $\phi$. Then for any line $\gamma\in\mathcal{B}$ not carried by $\mathcal{A}_{na}(\Lambda^{\pm}_\phi)$ at least one of the following hold:
\begin{enumerate}
 \item $\gamma$ is attracted to $\Lambda^+_\phi$ under iterations of $\phi$.
   \item $\gamma$ is attracted to $\Lambda^-_\phi$ under iterations of $\phi^{-1}$.
\end{enumerate}
Moreover, if $V^+_\phi$ and $V^-_\phi$ are attracting neighborhoods for the laminations $\Lambda^+_\phi$ and $\Lambda^-_\phi$ respectively, there exists an integer $l\geq0$ such that at least one of the following holds:
\begin{itemize}
 \item $\gamma\in V^-_\phi$.
\item $\phi^l(\gamma)\in V^+_\phi$
\item $\gamma$ is carried by $\mathcal{A}_{na}(\Lambda^{\pm}_\phi)$. 
\end{itemize}

\end{lemma}

\begin{corollary}
\label{WATgeo}
 Let $\phi\in \out$ be exponentially growing and $\Lambda^{\pm}_\phi$ be dual lamination pair for $\phi$ such that $\phi$ fixes $\Lambda^+_{\phi}$ and $\phi^{-1}$ fixes $\Lambda^{-}_\phi $with attracting neighborhoods $V^{\pm}_\phi$. 
 Then there exists some integer $l$ such that for any line $\gamma$ in $\mathcal{B}$ one of the following occurs:
\begin{itemize}
 \item $\gamma\in V^-_\phi$.
\item $\phi^l(\gamma)\in V^+_\phi $.
\item$\gamma$ is carried by $\mathcal{A}_{na}(\Lambda^{\pm}_\phi)$ 
\end{itemize}

\end{corollary}
\begin{proof}
 Let $K$ be a positive integer such that $\phi^K$ is rotationless. Then by definition $\mathcal{A}_{na}(\Lambda^{\pm}_\phi)=\mathcal{A}_{na}(\Lambda^{\pm}_{\phi^K})$. Also $\phi$ fixes $\Lambda^+_{\phi}$  implies 
$\Lambda^{+}_\phi=\Lambda^+_{\phi^K}$ and the attracting neighborhoods $V^+_\phi$ and $V^+_{\phi^K}$ can also be chosen to be the same weak neighborhoods.
Then by Lemma \ref{WAT} we know that there exists some positive integer $m$ such that the conclusions of the Weak attraction theorem hold for $\phi^K$. Let $l:=mK$. This gives us the conclusions of the corollary.
Before we end we note that by definition of an attracting neighborhood $\phi(V^+_\phi)\subset V^+_\phi$ which implies that if $\phi^l(\gamma)\in V^+_\phi$, then $\phi^t(\gamma)\in V^+_\phi$ for all $t\geq l$.
\end{proof}

\begin{lemma}
 Suppose $\phi,\psi\in \out$ are two exponentially growing automorphisms with attracting laminations $\Lambda^+_\phi$ and $\Lambda^+_\psi$, respectively. If a generic leaf $\lambda\in\Lambda^+_\phi$ is in $\mathcal{B}_{na}(\Lambda^+_\psi)$ 
 then the whole lamination $\Lambda^+_\phi\subset \mathcal{B}_{na}(\Lambda^+_\psi)$. 
 
\end{lemma}

\begin{proof}
 Recall that a generic leaf is bi-recurrent. Hence, $\lambda\in \mathcal{B}_{na}(\Lambda^+_\psi)$ implies that $\lambda$ is either carried by $\mathcal{A}_{na}$ or it is a generic leaf of some element of $\mathcal{L}(\psi^{-1})$. 
 First assume that $\lambda$ is carried by $\mathcal{A}_{na}$. Then using the fact that the set of lines carried by $ \mathcal{B}_{na}(\Lambda^+_\psi)$ is closed in the weak topology 
 , we can conclude that $\Lambda^+_\phi$ is carried by $\mathcal{A}_{na}(\Lambda^+_\psi)$.
 
 Alternatively, if $\lambda$ is a generic leaf of some element $\Lambda^-_\psi\in \mathcal{L}(\psi^{-1})$, then the weak closure $\overline{\lambda} = \Lambda^+_\phi = \Lambda^-_\phi$ and we know $\Lambda^-_\psi$ does not get attracted to $\Lambda^+_\psi$. 
 Hence, $\Lambda^+_\phi\subset\mathcal{B}_{na}(\Lambda^+_\psi)$.
\end{proof}

\textbf{Notations: }
\begin{itemize}
 \item Given a relative train track map $f:G\to G$ and a finite subpath $\beta\subset G$, by $N(G,\beta)$ we denote the open 
 neighborhood of $\mathcal{B}(G)$ defined by $\beta$, i.e. the set of all paths, circuits, lines in $G$ which contain 
 $\beta$ as a subpath (upto reversal of orientation).
 \item For a finite path $\beta\subset G$, we chose a lift $\tilde{\beta}$ and a lift $\tilde{f}:\tilde{G}\to\tilde{G}$ and 
 define $\tilde{f}_{\#\#}(\tilde{\beta})\subset \tilde{f}_{\#}(\tilde{\beta})\subset \tilde{G}$ to be the intersection of all 
 paths $\tilde{f}_\#(\tilde{\gamma})$ where $\tilde{\gamma}$ ranges over all paths in $\tilde{G}$ which contain $\tilde{\beta}$ 
 as a subpath. Define $f_{\#\#}(\beta)$ to be the projected image of $\tilde{f}_{\#\#}(\tilde{\beta})$ in $G$. Note that 
 $f_{\#\#}(\beta)$ is independent of the choice of $\tilde{f}$ and $\tilde{\beta}$ since $\tilde{f}_{\#\#}(\tilde{\beta})$  
 depends equivariantly on those choices. The bounded cancellation implies that $f_{\#\#}(\beta)$ is obtained from 
 $f_\#(\beta)$ by removing some initial and terminal segments of uniformly bounded length. 
 \item One can similarly define $f_{\#\#}(\beta)$ for any homotopy equivalence $f:G_1\to G_2$. The purpose of this definition is made clear from 
 Lemma \ref{expgrowth} which is used to show exponential growth and construct an attracting neighborhood using our pingpong argument that is to 
 follow.
\end{itemize}

\begin{lemma}[\cite{BFH-00} Section 2.3]
\label{doublesharp}
 If $f:G\longrightarrow G$ is a train track map for an irreducible $\phi\in$Out($F_n$) and $\alpha$ is a path in some leaf $\lambda$ of $G$ such that $\alpha=\alpha_1\alpha_2\alpha_3$ is a decomposition into subpaths such that $|\alpha_1|,|\alpha_3|\geq2C $
where C is the bounded cancellation constant for the map $f$, then $f^k_\#(\alpha_2)\subset f^k_{\#\#}(\alpha)$ for all $k\geq0$.
\end{lemma}

\begin{proof}
 Let $\alpha$ be any path with a decomposition $\alpha=\alpha_1\alpha_2\alpha_3$. Take lifts to universal cover of $G$. If $\tilde{\gamma}$ is a path in $\tilde{G}$ that contains $\tilde{\alpha}$, then decompose $\tilde{\gamma}=\tilde{\gamma_1}\tilde{\alpha_2}\tilde{\gamma_3}$
such that $\tilde{\alpha_1}$ is the terminal subpath of $\tilde{\gamma_1}$ and $\tilde{\alpha_3}$ is the initial subpath of $\tilde{\gamma_3}$. 
 Following the proof of \cite{BFH-00} if $K=2C$ then $\tilde{\gamma}$ can be split at the endpoints of $\tilde{\alpha_2}$. Thus, $\tilde{f^k_\#}(\tilde{\gamma})=\tilde{f^k_\#}(\tilde{\gamma_1})\tilde{f^k_\#}(\tilde{\alpha_2})\tilde{f^k_\#}(\tilde{\gamma_3})$.
The result now follows from the definition of $f^k_{\#\#}(\alpha)$. 

\end{proof}
\begin{lemma}[\cite{HM-13c} Lemma 1.1]
 \label{expgrowth}
Let $f:G\rightarrow G$ be a homotopy equivalence representing $\phi\in \out$ such that there exists a finite path $\beta\subset G$ having the property that 
$f_{\#\#}(\beta)$ contains three disjoint copies of $\beta$. Then $\phi$ is exponentially growing and there exists a lamination $\Lambda\in\mathcal{L}(\phi)$ and a generic leaf $\lambda$ of $\Lambda\in\mathcal{L}(\phi)$ 
such that $\Lambda$ is $\phi$-invariant and $\phi$ fixes $\lambda$ preserving orientation, each generic leaf contains $f^i_{\#\#}(\beta)$ as a subpath 
for all $i\geq0$ and $N(G,\beta)$ is an attracting neighborhood for $\Lambda$.
\end{lemma}


\section{Relatively irreducible free subgroups}
\label{3}
For the sake of keeping this work concise, we will only briefly go through the definitions and the reader is requested to refer to the work of Handel and 
Mosher titled ``Subgroup decomposition in $\out$'' (part IV in particular \cite{HM-13d}). 

\begin{definition}
Let $\phi \in \out$ and   $\mathcal{F}$ be a proper $\phi-$invariant free factor system. Then $\phi$ is said to be \emph{fully-irreducible relative to $\mathcal{F}$} if 
every component of $\mathcal{F}$ is invariant under $\phi$ and there does 
not exist any proper $\phi$-periodic free factor system $\mathcal{F}'$ such that $\mathcal{F} \sqsubset \mathcal{F}'$ and $\mathcal{F} \ne \mathcal{F}'$.
\end{definition}

\begin{remark}
 \emph{This definition is mildly restrictive in the sense that we may have cases when $\phi$ is not fully irreducible rel $\mathcal{F}$ but some power of 
 $\phi$ is fully irreducible rel $\mathcal{F}$. This arises out of our requirement that $\phi$ leaves each component of $\mathcal{F}$ invariant. This requirement is 
 automatically satisfied by all rotationless elements and all $\text{IA}_n(\mathbb{Z}_3)$ elements used by Handel-Mosher in \cite{HM-13d}}.
\end{remark}

From Handel-Mosher's work on loxodromic elements of the free splitting complex \cite{HM-14b} one defines the concept of a \textit{co-edge} number for a free factor system $\mathcal{F}$ : it is an integer $\geq 1$ which is the minimum 
, over all subgraphs $H$ of a marked graph $G$ such that H realizes $\mathcal{F}$, of the number of edges in $G-H$ . Lemma 4.8 in \cite{HM-14b} gives an explicit formula for computing the co-edge 
number for a given free factor system. The following result is vital to our work here and will help us identify the peripheral subgroups 
in the proof of relative hyperbolicity of mapping tori of outer automorphisms which are fully irreducible relative to a free factor system 
\ref{relhyp2}. The result is present implicitly in \cite{HM-13c}, but for sake of clarity we put it down as a lemma here.

\begin{lemma}\label{firelf}
 Suppose $\mathcal{F}\sqsubset \{[\F]\}$ is a multi-edge (i.e. coedge number $\geq 2$) extension invariant under $\phi$ and every component of $\mathcal{F}$ is also 
 $\phi-$invariant. 
 If $\phi$ is fully irreducible rel $\mathcal{F}$ then there exists $\phi-$invariant dual lamination pair $\Lambda^\pm_\phi$ such that the following hold:
 
 \begin{enumerate}
 \item $\Lambda^\pm_\phi$ fills relative to $\mathcal{F}$.
  \item If $\Lambda^\pm_\phi$ is nongeometric then  $\mathcal{A}_{na}(\Lambda^\pm_\phi)=\mathcal{F}$.
  \item If $\Lambda^\pm_\phi$ is geometric then there exists a root free $\sigma\in\F$ such that 
  $\mathcal{A}_{na}(\Lambda^\pm_\phi) = \mathcal{F}\cup \{[\langle \sigma \rangle]\}$.
 \end{enumerate}

 Conversely, if there exists a $\phi-$invariant dual lamination pair such that if (1) and (2) hold  or if (1) and (3) hold then $\phi$ is fully irreducible rel $\mathcal{F}$.
 
\end{lemma}

\begin{proof}
Let $f: G\to G$ be an improved relative train track map representing $\phi$ and $G_{r-1}$ be the filtration element realizing $\mathcal{F}$.
 Apply \cite[Proposition 2.2]{HM-13d}  to get all the conclusions for some iterate $\phi^k$ of $\phi$. 
 
 We claim that $\Lambda^\pm_\phi$ obtained by applying \cite[Proposition 2.2]{HM-13d} must be $\phi-$invariant.
 Otherwise, by using the definition of ``fully irreducible relative to $\mathcal{F}$'' we conclude that 
 $\phi(\Lambda^+_\phi)$ will be an attracting lamination which is properly contained in $G_{r-1}$ and hence is carried by $\mathcal{F}$ which in turn 
 is carried by the nonattracting subgroup system for $\Lambda^+_\phi$. This is a contradiction, hence $\Lambda^+_\phi$ is $\phi-$invariant. Similar arguments 
 work for $\Lambda^-_\phi$.
 
 The converse part follows from  the case analysis in the proof of \cite[Theorem I, pages 18-19]{HM-13d}

\end{proof}

In view of the above lemma we introduce the following terminologies, originally due to \cite{HM-13c}, that we will be using for the rest of the paper. 
Suppose $\mathcal{F}\sqsubset \{[\F]\}$ is a multi-edge (i.e. coedge number $\geq 2$) extension invariant under $\phi$. 

 \begin{description}
  \item $\Lambda^\pm$ \textbf{\emph{fills relative to}} $\mathcal{F}$ simply means that $\mathcal{F}_{supp}(\mathcal{F}, \Lambda^\pm) = [\F]$. 
  \item We say that $\phi$ is fully irreducible rel $\mathcal{F}$ and \textbf{\emph{nongeometric above $\mathcal{F}$}} if and only if there exists a
  $\phi-$invariant dual lamination pair $\Lambda^\pm_\phi$ such that items (1) and (2) are 
  satisfied in Lemma \ref{firelf}.
  \item We say that $\phi$ is fully irreducible rel $\mathcal{F}$ and \textbf{\emph{geometric above $\mathcal{F}$}} if and only if 
  there exists a $\phi-$invariant dual lamination pair $\Lambda^\pm_\phi$ such that items (1) and (3) are 
  satisfied in Lemma \ref{firelf}.
 \end{description}

 Our objective here is to show that if $\phi$ is fully irreducible relative to a multi-edge extension $\mathcal{F}$ then the mapping torus of $\phi$ is 
 strongly hyperbolic relative to the collection of subgroups that define the nonattracting subgroup system $\mathcal{A}_{na}(\Lambda^\pm_\phi)$. One of 
 our assumptions in the definition of being fully irreducible relative to $\mathcal{F}$ is that $\phi$ fixes every component of $\mathcal{F}$. This is done 
 purely for the simplicity of the proof. If $\phi$ permutes the components of $\mathcal{F}$ then one can pass to a finite power of $\phi$, call it $\phi'$, 
 and assume that $\phi'$ stabilizes each component of $\mathcal{F}$. The strong relative hyperbolicity of mapping torus of $\phi'$ then 
 implies the strong relative hyperbolicity of $\phi$ by using Drutu's theorem \cite{Dr-09}.

 We now state the Relativized version of the pingpong lemma. This lemma is a modification of the pingpong proposition (proposition 4.4) in \cite{self1} and the proof is very  similar.
 The lemma that has been proven by Handel and Mosher in Proposition 1.3 in \cite{HM-13d} is a special case of the following proposition. What they have shown (with slightly weaker conditions) 
 is that the lemma is true for $k=1$ and only under positive 
powers of $\psi$ and $\phi$. Strengthening hypothesis slightly enables us to extend their result to both positive and negative exponents and also for reduced words with arbitrary $k$ (see statement of \ref{pingpong} for description of $k$). 
They also have the assumption that $\phi,\psi$ are both 
rotationless, which they later on discovered, is not necessary; 
one can get away with laminations that are left invariant. The main technique, however, is same.

 However, there are some small changes in the statement and a subtle change in the proof where we need to use 
 a modified version of relative train track maps in \cite{FH-11} to accommodate for $\mathcal{F}$.

 \begin{proposition}
 \label{pingpong}
 Let $\mathcal{F}$ be a proper free factor system that is invariant under $\phi, \psi \in \out$. Let $\Lambda^\pm_\phi, \Lambda^\pm_\psi$ be invariant dual lamination pairs for $\phi, \psi$ respectively.
   Suppose that the laminations $\Lambda^\pm_\phi, \Lambda^\pm_\psi$ each have a generic leaf $\lambda^\pm_\phi, \lambda^\pm_\psi$ which is fixed by 
  $\phi^\pm, \psi^\pm$ respectively, with fixed orientation. Also assume that the following conditions hold:
  
  \begin{itemize}
   \item $\Lambda^\pm_\phi$ is weakly attracted to $\Lambda^\epsilon_\psi$ under iterates of $\psi^\epsilon$ (where $\epsilon = +, -$).
   \item $\Lambda^\pm_\psi$ is weakly attracted to $\Lambda^\epsilon_\phi$ under iterates of $\phi^\epsilon$ (where $\epsilon = +, -$).
   \item $\mathcal{F} \sqsubset \mathcal{A}_{na}(\Lambda^\pm_\phi)$ and $\mathcal{F} \sqsubset \mathcal{A}_{na}(\Lambda^\pm_\psi)$
   \item Either both the lamination pairs $\Lambda^\pm_\psi, \Lambda^\pm_\psi$ are non-geometric or the subgroup $\langle \phi, \psi \rangle$ 
	    is geometric above $\mathcal{F}$
  \end{itemize}

 Then there exist attracting neighborhoods $V^\pm_\phi, V^\pm_\psi$ of $\Lambda^\pm_\phi, \Lambda^\pm_\psi$ respectively, 
 and there exists an integer $M$,  such that for every pair of finite sequences $n_i\geq M$ and $m_i\geq M$ if

 $$\xi= \psi^{\epsilon_1 m_1}\phi^{\epsilon_1 ' n_1}........\psi^{\epsilon_k m_k}\phi^{\epsilon_k ' n_k}$$ 
$(k\geq 1)$ is a cyclically reduced word then $\xi$ will be exponentially-growing and have a $\xi-$invariant dual lamination pair $\Lambda^\pm_\xi$ satisfying the following properties:
  
  \begin{enumerate}
   \item $\Lambda^\pm_\xi$ is non-geometric if $\Lambda^\pm_\phi$ and $\Lambda^\pm_\psi$ are both non-geometric.
   \item $\mathcal{F}$ is carried by $\mathcal{A}_{na}(\Lambda^\pm_\xi)$ and  $\mathcal{A}_{na}(\Lambda^\pm_\xi)$ is caried by  $\mathcal{A}_{na}(\Lambda^\pm_\phi)$ and  
   $\mathcal{A}_{na}(\Lambda^\pm_\psi)$.
    \item $\psi^{m_i}(V^\pm_\phi) \subset V^+_\psi$ and $\psi^{-m_i}(V^\pm_\phi) \subset V^-_\psi$ .
  \item $\phi^{n_j}(V^\pm_\psi) \subset V^-_\phi$ and $\phi^{-n_j}(V^\pm_\psi) \subset V^-_\phi$ .
  \item $V^+_\xi \colon= V^{\epsilon_1}_\psi$ is an attracting neighborhood of $\Lambda^+_\xi$ 
  \item $V^-_\xi \colon= V^{-\epsilon_k '}_\phi$ is an attracting neighborhood of $\Lambda^-_\xi$ 
  \item (uniformity) Suppose $U^{\epsilon_1}_\psi$ is an attracting neighborhood of $\Lambda^{\epsilon_1}_\psi$ then some generic leaf of $\Lambda^+_\xi$ belongs to $U^{\epsilon_1}_\psi$ for 
  sufficiently large $M$.
  \item (uniformity) Suppose $U^{\epsilon_k'}_\phi$ is an attracting neighborhood of $\Lambda^{\epsilon_k'}_\phi$ then some generic leaf of $\Lambda^+_\xi$ belongs to $U^{\epsilon_k'}_\phi$ for 
  sufficiently large $M$.
  \end{enumerate}

 \end{proposition}
 
 \textbf{Notation: }
 \begin{itemize}
  \item For convenience, we shall use $\Lambda_0^\epsilon$ to denote attracting ($\epsilon=+$) or repelling ($\epsilon=-$) 
lamination of $\phi$  and $\Lambda_1^\epsilon$ to denote attracting or repelling lamination of $\psi$ which are given in the hypothesis.
 \item Also for easier book-keeping for our relative train-track maps and homotopy equivalences we use the notation $\mu_i=(i,\epsilon)$ (where $i\in\{0, 1\}$ and 
 $\epsilon\in\{+, -\}$) in the following way:
 
 $g_{\mu_i}:G_{\mu_i}\longrightarrow G_{\mu_i}$ denotes the improved relative train-track map for $\phi$ if $\mu_i=(0,+)$. At first glance it may seem like a strange choice 
 but there are a lot of superscripts and subscripts which appear in the proof and this choice helps to keep the proof as clean as 
 possible. To avoid any confusion we have made sure to clearly mention what $\mu_i$ is during every use.
 
 \end{itemize}

\begin{proof}

Let $g_{\mu_i}:G_{\mu_i}\longrightarrow G_{\mu_i}$ be improved relative train train-track maps and  $u^{\mu_i}_{\mu_j}:G_{\mu_i}\longrightarrow G_{\mu_j}$ be the homotopy equivalence between the graphs which preserve the markings, 
where $i\neq j$ (where $i\in\{0, 1\}$ and $\epsilon\in\{+, -\}$).
Also suppose (by Theorem 5.1.5, \cite{BFH-00} ) that $\mathcal{F}$ is realized by some filtration element.

Let $C_1> 2\text{BCC}\{g_{\mu_i}|i \in\{0,1\}\}$. Let $C_2> \text{BCC}\{u^{\mu_i}_{\mu_j}| i,j\in\{0,1\},i\neq j\}$. Let
 $C\geq C_1,C_2$.

Now we work with  $\lambda_i^{\epsilon}$ as generic leaves of laminations $\Lambda_i^{\epsilon}$.
\paragraph{Step 1:}
Using the fact that $\Lambda_1^{\epsilon}$ is weakly attracted to $\Lambda_0^{\epsilon '}$, under the action if $\psi^{\epsilon '}$, choose a finite subpath $\alpha_1^{\epsilon}\subset\lambda^{\epsilon}_1$ such that 
\begin{itemize}
 \item $(u_{\mu_0}^{\mu_1})_{\#}(\alpha_1^{\epsilon})\rightarrow \lambda^{\epsilon '}_0$ weakly, where $\mu_0=(0,\epsilon ')$ and $\mu_1=(1,\epsilon)$.
\item $\alpha_1^{\epsilon}$ can be broken into three segments: initial segment of $C$ edges, followed by a subpath $\alpha^{\epsilon}$ followed by a terminal segment with $C$ edges.
\end{itemize}

\paragraph{Step 2:}
 Now using the fact $\Lambda^{\epsilon}_0\rightarrow \Lambda^{\epsilon '}_1$ weakly, under iterations of $\phi^{\epsilon '}$, we can find positive integers $p^{\epsilon}_{\mu_1}$ (there are four choices here that will yield four integers) such that 
$\alpha^{\epsilon'}_1 \subset (g^{p^{\epsilon}_{\mu_1}}_{\mu_1} u_{\mu_1}^{\mu_0})_\#(\lambda^{\epsilon}_0)$ , where $\mu_0=(0,\epsilon) , \mu_1=(1,\epsilon')$. \\
Let $C_3$ be greater than BCC$\{g^{p^{\epsilon}_{\mu_1}}_{\mu_1} u_{\mu_1}^{\mu_0}\} $ (four maps for four integers $p^{\epsilon}_{\mu_1}$) .\\

\paragraph{Step 3:}
Next, let $\beta^{\epsilon}_0\subset \lambda^{\epsilon}_0$ be a finite subpath such that $(g^{p^{\epsilon}_{\mu_1}}_{\mu_1})_\#(\beta^{\epsilon}_0)$ contains $\alpha^{\epsilon '}_1$ protected by $C_3$ edges in both sides, where $\mu_0=(0,\epsilon)$ and $\mu_1=(1,\epsilon ')$.
Also, by increasing $\beta^{\epsilon}_0$ if necessary, we can assume that $V^{\epsilon}_\psi = N(G_{\mu_0},\beta^{\epsilon}_0)$ is an attracting neighborhood of $\Lambda^{\epsilon}_0$ .

Let $\sigma$ be any path containing $\beta^{\epsilon}_0$. Then $(g^{p^{\epsilon}_{\mu_1}}_{\mu_1} u_{\mu_1}^{\mu_0})_\#(\sigma)\supset \alpha^{\epsilon}_0$. Thus by using Lemma \ref{doublesharp} we get that  $(g^{p^{\epsilon}_{\mu_1}+t}_{\mu_1} u_{\mu_0}^{\mu_1})_\#(\sigma)=(g^t_{\mu_1})_\#((g^{p^{\epsilon}_{\mu_1}}_{\mu_1} u_{\mu_1}^{\mu_0})_\#(\sigma))$ contains $(g^t_{\mu_1})_\#(\alpha^{\epsilon})$ for all $t\geq 0$ .\\
Thus we have $(g^{p^{\epsilon}_{\mu_1}+t}_{\mu_1} u_{\mu_1}^{\mu_0})_{\#\#}(\beta^{\epsilon}_0)\supset (g^t_{\mu_1})_\#(\alpha^{\epsilon})$ for all $t\geq0$.

This proves conclusion (3)

\paragraph{Step 4:}
Next step is reverse the roles of $\phi$ and $\psi$ to obtain positive integers $q^{\epsilon'}_{\mu_1}$ and paths $\gamma^{\epsilon'}_1\subset\lambda^{\epsilon'}_{\mu_1}$ such that $(g^{q^{\epsilon}_{\mu_0}+t}_{\mu_0} u_{\mu_0}^{\mu_1})_{\#\#}(\gamma^{\epsilon'}_1)\supset (g^t_{\mu_0})_\#(\beta^{\epsilon}_0)$ for all $t\geq0$ 
, where $\mu_0=(0,\epsilon), \mu_1= (1,\epsilon')$ .

This proves conclusion (4)

\paragraph{Step 5:}
Finally, let $k$ be such that $(g^k_{\mu_1})_{\#}(\alpha^{\epsilon})$ contains three disjoint copies of $\gamma^{\epsilon}_1$ and that $(g^{k}_{\mu_0})_{\#}(\beta^{\epsilon}_0)$ contains three 
disjoint copies of $\beta^{\epsilon}_0$ for $\epsilon = 0,1$. Let $p\geq$ max $\{ p^{\epsilon}_{\mu_1} \}$  + $k$ and $q\geq$ max $\{q^{\epsilon}_{\mu_0} \}$  + $k$.

Let $m_i\geq q$ and $n_i\geq p$.

The map $f_\xi = g^{m_1}_{(0,\epsilon_1)}u^{(1,\epsilon_1')}_{(0,\epsilon_1)}g^{n_1}_{(1,\epsilon_1')}u^{(0,\epsilon_2)}_{(1,\epsilon_1')}.........g^{n_k}_{(1,\epsilon_k ')}u^{(0,\epsilon_1)}_{(1,\epsilon_k ')}:G_{(0,\epsilon_1)}\rightarrow G_{(0,\epsilon_1)}$ is a topological representative of $\xi$.
With the choices we have made,  $g^{n_k}_{(1,\epsilon_k')}u^{(0,\epsilon_1)}_{(1,\epsilon_k')})_{\#\#}(\beta^{\epsilon_1}_0)$ contains three disjoint copies of $\gamma_1^{\epsilon_k'}$ and so 
$(g^{m_k}_{(0,\epsilon_k)}u^{(1,\epsilon_k')}_{(0,\epsilon_k)}g^{n_k}_{(1,\epsilon_k')}u^{(0,\epsilon_1)}_{(1,\epsilon_k')})_{\#\#}(\beta^{\epsilon_1}_0) $ will contain three disjoint copies of $\beta_0^{\epsilon_k}$. 
Continuing in this fashion in the end we get that $(f_\xi)_{\#\#}(\beta_0^{\epsilon_1})$ contains three disjoint copies of $\beta_0^{\epsilon_1}$. Thus by Lemma \ref{expgrowth} $\xi$ is an exponentially growing element of $\out$ with an attracting lamination $\Lambda^+_\xi$ which has 
$V^+_\xi =  N(G_{\mu_0}, \beta^{\epsilon_1}_0) = V^{\epsilon_1}_\psi$ as an attracting neighborhood. This proves conclusion (5).

Similarly, if we take inverse of $\xi$ and interchange the roles played by $\psi, \phi$ with $\phi^{-1},\psi^{-1}$, we can produce an attracting lamination $\Lambda^-_{\xi}$ for $\xi^{-1}$ with an attracting neighborhood $V^-_\xi = N(G_{(1,-\epsilon_k')},\gamma_1^{-\epsilon_k'})=V^{-\epsilon_k'}_\phi$, 
 which proves  property (5) and (6) of the proposition.
The proof in  Proposition 1.3 \cite{HM-13d} that $\Lambda^-_\xi$ and $\Lambda^+_\xi$ are dual lamination pairs will carry over in this situation and so $\mathcal{A}_{na}\Lambda^+_\xi=\mathcal{A}_{na}\Lambda^-_\xi$ .

Hence, every reduced word of the group $\langle\phi^n,\psi^m\rangle$ will be exponentially growing if $n\geq p,m\geq q$. Let $M\geq p,q$. 

Now, we prove the conclusion (1) and (2) related to non-attracting subgroup system $\mathcal{A}_{na}(\Lambda^\pm_\xi)$.
By corollary \ref{WATgeo}  there exists $l$ so that if $\tau$ is neither an element of $V^{-\epsilon_k'}_\phi = V^-_\xi$ nor is it carried by $\mathcal{A}_{na}\Lambda^\pm_\phi$ then $\phi^{\epsilon_k't}_\#(\tau)\in V^{\epsilon_k'}_\phi$ for all $t\geq l$ . Increase $M$
 if necessary so that $M>l$. Under this assumption, 
$\xi_\#(\tau)\in \psi^{\epsilon_1m_1}(V^{\epsilon_1'}_\phi)\subset V^+_\xi$. So $\tau$ is weakly attracted to $\Lambda^+_\xi$. Hence we conclude that if $\tau\notin V^-_\xi$ and not attracted to $\Lambda_\xi^+$, then $\tau$ is  carried by $\mathcal{A}_{na}\Lambda^\pm_\phi$.
Similarly, if $\tau$ is not in $V^+_\xi$ and not attracted to $\Lambda^-_\xi$ then $\tau$ is carried by $\mathcal{A}_{na}\Lambda^\pm_\psi$. 

For the remainder of the proof, for notational simplicity, assume that $\xi$ begins with a positive power of $\psi$ and ends with a positive power of $\phi$. 

Next, suppose that $\tau$ is a line that is not attracted to any of $\Lambda^+_\xi, \Lambda^-_\xi$. Then $\tau$ must be disjoint from $V^+_\xi, V^-_\xi$. So, is carried by both $\mathcal{A}_{na}\Lambda^\pm_\psi$ and $\mathcal{A}_{na}\Lambda^\pm_\phi$.
Restricting our attention to periodic line, we can say that every conjugacy class that is carried by both $\mathcal{A}_{na}\Lambda^+_\xi$ and $\mathcal{A}_{na}\Lambda^-_\xi$ is carried by both $\mathcal{A}_{na}\Lambda^\pm_\psi$ and $\mathcal{A}_{na}\Lambda^\pm_\phi$. 
Since every line carried by $\mathcal{A}_{na}(\Lambda^\pm_\xi)$ is a limit of conjugacy classes which are carried by 
$\mathcal{A}_{na}(\Lambda^\pm_\xi)$, we can conclude that every line carried by $\mathcal{A}_{na}(\Lambda^\pm_\xi)$ is carried by both 
$\mathcal{A}_{na}(\Lambda^\pm_\phi)$ and $\mathcal{A}_{na}(\Lambda^\pm\psi)$. Since a generic leaf of $\Lambda^+_\psi$ realized in $G_\psi$ contains the 
finite path $\beta$ which begins and ends with edges contained in $H_\psi$, and since $\mathcal{F}\sqsubset\mathcal{A}_{na}(\Lambda^\pm_\psi)$
the filtration element of $G_\psi$ corresponding to $\mathcal{F}$ is below the stratum $H_\psi$. This implies that a generic leaf of 
$\Lambda^+_\xi$ is not carried by $\mathcal{F}$. Also note that since $\xi$ fixes $\mathcal{F}$, the limit of any conjugacy 
class under iterates of $\xi$ is contained in $\mathcal{F}$ if the conjugacy class is itself carried by $\mathcal{F}$. This implies that 
$\mathcal{F}$ does not carry $\Lambda^+_\xi$. Hence no conjugacy class carried by $\mathcal{F}$ is attracted to 
$\Lambda^+_\xi$ under iterates of $\xi$. Hence $\mathcal{F}\sqsubset\mathcal{A}_{na}(\Lambda^\pm_\xi)$. This proves (2).

It remains to show that if $\Lambda^\pm_\psi$ and $\Lambda^\pm_\psi$ are nongeometric then $\Lambda^\xi$ is also nongeometric. 
Suppose on the contrary that $\Lambda^\pm_\xi$ is geometric. Then by using Proposition 2.18 from \cite{HM-13a} we can conclude that there 
exists a finite set of conjugacy classes, which is fixed by $\xi$, and the free factor support of this set of conjugacy class carries the lamination $\Lambda^+_\xi$.
But by using (2), all $\xi-$invariant conjugacy classes are carried by $\mathcal{A}_{na}(\Lambda^\pm_\xi)$ and hence carried by both 
$\mathcal{A}_{na}(\Lambda^\pm_\phi)$ and $\mathcal{A}_{na}(\Lambda^\pm_\psi)$. This implies that $\Lambda^\pm_\xi$ 
is carried by $\mathcal{A}_{na}(\Lambda^+_\psi)$ and $\mathcal{A}_{na}(\Lambda^\pm_\phi)$. But the realization of a generic leaf of 
$\Lambda^+_\xi$ in $G_\psi$ contains the subpath $\beta$ and hence cannot be carried by $\mathcal{A}_{na}(\Lambda^\pm_\psi)$
, a contradiction. This completes the proof of (1).

\end{proof}

\begin{remark}
\begin{enumerate}
 \item Notice that if $\phi$ and $\psi$ are both hyperbolic outer automorphisms then, 
 $\Lambda^\pm_\psi$ and $\mathcal{A}_{na}(\Lambda^\pm)_\phi$ are both nongeometric and we always satisfy the fourth condition 
 in our list of assumptions (the presented under the bullets). Thus the above proposition is true for hyperbolic $\phi$ and $\psi$ 
 without the assumption in fourth bullet. 
 \item Any 
 $\xi-$invariant conjugacy class must be carried by $\mathcal{A}_{na}(\Lambda^\pm_\xi)$ and hence carried by both 
 $\mathcal{A}_{na}(\Lambda^\pm_\phi)$ and $\mathcal{A}_{na}(\Lambda^\pm_\psi)$. Hence, if we assume that these two subgroup 
 systems are mutually malnormal relative to $\mathcal{F}$, we can conclude that every $\xi-$periodic conjugacy class is 
 carried by $\mathcal{F}$. This will be very useful when we deduce fully-irreducibility for $\xi$ in the next theorem.

 \end{enumerate}
 
\end{remark}

\begin{definition} \label{relind}

 Let $\phi, \psi \in \out$ be exponentially growing outer automorphisms with invariant lamination pairs $\Lambda^\pm_\phi, \Lambda^\pm_\psi$ and let $\mathcal{F}$ be a free factor system 
 which is left invariant by both $\phi, \psi$. Suppose the following conditions hold:
 \begin{enumerate}
  \item None of the lamination pairs $\Lambda^\pm_\phi, \Lambda^\pm_\psi$ are carried by $\mathcal{F}$.
  \item $\{\Lambda^\pm_\phi\} \cup \{\Lambda^\pm_\psi\}$ fill relative to $\mathcal{F}$.
  \item $\Lambda^\pm_\phi$ is weakly attracted to $\Lambda^\epsilon_\psi$ under iterates of $\psi^\epsilon$ (where $\epsilon = +, -$).
  \item $\Lambda^\pm_\psi$ is weakly attracted to $\Lambda^\epsilon_\phi$ under iterates of $\phi^\epsilon$ (where $\epsilon = +, -$).
  \item $\mathcal{A}_{na}(\Lambda^\pm_\phi)$ and $\mathcal{A}_{na}(\Lambda^\pm_\psi)$ are mutually malnormal relative to $\mathcal{F}$.
  \item Either both the lamination pairs $\Lambda^\pm_\psi, \Lambda^\pm_\psi$ are non-geometric or the subgroup $\langle \phi, \psi \rangle$ 
	    is geometric above $\mathcal{F}$
 \end{enumerate}

 In this case we define the pair $(\phi, \Lambda^\pm_\phi) , (\psi, \Lambda^\pm_\psi)$ to be independent \textit{relative to $\mathcal{F}$}.
\end{definition}

\textbf{Remark:} Here the term \textit{mutually malnormal relative to $\mathcal{F}$} in condition 5 above, is by definition the property that a line or a 
conjugacy class is carried by both 
$\mathcal{A}_{na}(\Lambda^\pm_\phi)$ and $\mathcal{A}_{na}(\Lambda^\pm_\psi)$ if and only if it is carried by $\mathcal{F}$.

 \begin{thma}
  Given a free factor system $\mathcal{F}$ with co-edge number $\geq 2$, given $\phi, \psi \in \out$ each preserving $\mathcal{F}$ and its components individually, and given invariant lamination pairs 
  $\Lambda^\pm_\phi, \Lambda^\pm_\psi$, so that the pair $(\phi, \Lambda^\pm_\phi), (\psi, \Lambda^\pm_\psi)$ is independent relative to $\mathcal{F}$, then there $\exists$  $M\geq 1$, 
  such that for any integer $m,n \geq M$, the group $\langle \phi^m, \psi^n \rangle$ is a free group of rank 2, all of whose non-trivial elements except perhaps the powers of $\phi, \psi$ 
  and their conjugates, are fully irreducible relative to $\mathcal{F}$ with a lamination pair which fills relative to $\mathcal{F}$. 
  
  In addition if both $\Lambda^\pm_\phi, \Lambda^\pm_\psi$ are non-geometric 
  then this lamination pair is also non-geometric.
 \end{thma}

 \begin{proof}
  The conclusion about the rank 2 free group follows easily from the Tit's Alternative work of Bestvina-Feighn-handel \cite{BFH-00} Lemma 3.4.2, which gives us some integer $M_0$ such that 
  for every $m,n \geq M_0$ the group $\langle \phi^m, \psi^n \rangle$ is a free group of rank 2.
  
  For the conclusion about being fully irreducible relative to $\mathcal{F}$, suppose that the conclusion is false.

	$\Longrightarrow$ For every $M \geq M_0$, there exists some $m(M), n(M)$ such that the group $\langle \phi^m, \psi^n \rangle$ contains at least one 
	non trivial reduced word $\xi_M$ which is not the powers of generators themselves or their conjugates, and $\xi_M$ is not fully irreducible relative to $\mathcal{F}$.

	Next, using the conclusions from our relativized pingpong lemma earlier in this section and by increasing $M_0$ if necessary,
	we can conclude that $\mathcal{F}$ is carried by $\mathcal{A}_{na}(\Lambda^\pm_{\xi_M})$ and  $\mathcal{A}_{na}(\Lambda^\pm_{\xi_M})$ is caried by  $\mathcal{A}_{na}(\Lambda^\pm_\phi)$ and  
   $\mathcal{A}_{na}(\Lambda^\pm_\psi)$.
   
	$\Longrightarrow$ $\mathcal{A}_{na}(\Lambda^\pm_{\xi_M}) = \mathcal{F}$.
	
	$\Longrightarrow$ $\Lambda^\pm_{\xi_M}$ is non-geometric.
	
	Also the additional co-edge $\geq 2$ condition tells us that $\Lambda^\pm_{\xi_M}$ cannot fill relative to $\mathcal{F}$. This means that the 
	free factor system $\mathcal{F}_M := \mathcal{F}_{supp}(\mathcal{F}, \Lambda^\pm_{\xi_M})$ is a proper free factor system and that $\mathcal{F}_M\sqsubset \mathcal{F}$ is a 
	proper containment for all sufficiently large $M$.

 We also make an assumption that this $\xi_M$ begins with a nonzero power of $\psi$ and ends in some nonzero power of $\phi$; if not, then we can conjugate to achieve this.
Thus as $M$ increases, we have a sequence of reducible elements $\xi_M\in \out$. Pass to a subsequence to assume that the $\xi_M$'s begin with a positive power of $\psi$ and end with a positive power of $\phi$. If no such subsequence exist, 
then change the generating set of by replacing generators with their inverses.

Let $\xi_M=\psi^{m_1}\phi^{\epsilon'_1 n_1}........\psi^{\epsilon_k m_k}\phi^{n_k}$ \hspace{2pc}where $m_i=m_i(M),n_j=n_j(M)$ and $k$ depend on $M$.

We note that by our assumptions, the exponents get larger as $M$ increases and in this setting we can draw a conclusion about weak attracting neighborhoods of the laminations $\lambda^\pm_{\xi_M}$. 
From the Ping-Pong lemma we know that there exists attracting neighborhoods $V^\pm_\psi$  and $V^\pm_\phi$ for  the dual lamination pairs $\Lambda^\pm_\psi$ and $\Lambda^\pm_\phi$ , respectively,  such that if $i\neq 1$ $$\psi^{{\epsilon_i m_i}(M)}(V^\pm_\phi)\subset V^{\epsilon_i}_\psi \textrm{ and } \psi^{m_1(M)}(V^\pm_\phi)\subset V^+_\psi\subset V^+_{\xi_M}$$
where each of $\xi_M$'s are exponentially-growing and equipped with a lamination pair $\Lambda^\pm_{\xi_M}$ (with attracting neighborhoods $V^\pm_{\xi_M}$) such that $\mathcal{A}_{na}\Lambda^\pm_{\xi_M}$ is trivial (using conclusion 1 of proposition \ref{pingpong} and bullet 6 in the hypothesis set).
	

	Now we proceed following the key idea of proof of Theorem I in the non-geometric case as in section 2.4 of \cite{HM-13d}. The idea is to use 
	stallings graph to drive up $\mathcal{F}_M $ and arrive at a contradiction similar to the proof Theorem 5.7 in \cite{self1}.
	This is achieved in our proof by showing that if $M$ is sufficiently large then, we have $\mathcal{F}_\phi , \mathcal{F}_\psi \sqsubset \mathcal{F}_M$ and so $\mathcal{F}_{supp}(\mathcal{F}_\phi, \mathcal{F}_\psi)\sqsubset \mathcal{F}_M$
	(where $\mathcal{F}_\phi := \mathcal{F}_{supp}(\mathcal{F}, \Lambda^\pm_\phi)$ and $\mathcal{F}_\psi := \mathcal{F}_{supp}(\mathcal{F}, \Lambda^\pm_\psi)$). 
		This will imply that $\mathcal{F}_{supp}(\mathcal{F}_\phi, \mathcal{F}_\psi)\sqsubset \mathcal{F}$ is a proper containment, which will contradict the condition that $\{\Lambda^\pm_\phi\} \cup \{\Lambda^\pm_\psi\}$ fill relative to $\mathcal{F}$.
	
	To proceed with the proof, fix a marked metric graph $H$ with a core subgraph $H_0$ realizing $\mathcal{F}$.
	For each $M$ , let $[F_M]$ denote the component of $\mathcal{F}_M$ that supports $\lambda^\pm_{\xi_M}$ (since the free factor support of 
	$\Lambda^\pm_{\xi_M}$ is a single free factor). Note that the free factor system $\mathcal{F}\wedge[F_M]$ is exactly the set set of components 
	of $\mathcal{F}$ that are contained in $[F_M]$. 
	Denote the stallings graph associated to $F_M$ by $K_M$ (which is core of covering space of $H$ associated to the subgroup $F_M$), 
	equipped with the immersion $p_M:K_M \rightarrow H$ such that $[p_*(\pi_1(K_M))] = [F_M]$. Since $F_M$ is a free factor of $\F$, we can embed $K_M$ inside a marked graph $G_M$ such that the map 
	$p_M$ lifts to a homotopy equivalence $q_M:G_M\to H$ that preserves the marking. In this setup, a line $\gamma$ is carried by 
	$[F_M]$ if and only if it is contained in $K_M$ and this implies that the leaves of the laminations $\Lambda^+_{\xi_M}$ and 
	$\Lambda^-_{\xi_M}$ are contained in $K_M$ and $q_M$ restricted to any such leaf is an immersion whose image is the 
	realization of $\gamma$ in $H$.

	A \textit{natural vertex} of $K_M$ is a vertex with valence greater than two and a \textit{natural edge} is an edge between two natural vertices. We can subdivide every natural edge of $K_M$ into \textit{edgelets}, 
	so that each edgelet is mapped to an edge in $H$ and label the edgelet by its image in $H$. There is a unique subgraph $\widehat{K}_M$ of $K_M$ which keeps track of the free factor components of 
	$\mathcal{F}\wedge[F_M]$. For such a subgraph, the restriction $p|_{\widehat{K}_M}$ is a homeomorphism onto the components of 
	$H_0$ corresponding to $\mathcal{F}\wedge[F_M]$. In fact, this restriction is a cellular isomorphism at the level of edgelets.
	This gives us  $$[K_M]= [F_M]= \mathcal{F}_{supp}(\mathcal{F}\wedge[F_M], \Lambda^\pm_{\xi_M})= \mathcal{F}_{supp}(\widehat{K}_M, \Lambda^\pm_{\xi_M})$$

Let $\gamma_M^-$ be a generic leaf of $\Lambda^-_M$ and $\gamma^+_M$ be a generic leaf of $\Lambda^+_M$. The realizations of these leaves in $H$  is contained in the subgraph $K_M$ and the above relation tells us that 
every natural edge of $K_M$ that is not contained in $\widehat{K}_M$ is crossed by both $\gamma^+_M$ and $\gamma^-_M$. We use this observation to 
prove a property of edgelets in $K_M$.
For each integer $C>0$ consider the set $Y_{M,C} \subset K_M$ to be the $C-$ neighborhood of the set of 
natural vertices of $K_M$ and put a path metric on $Y_{M,C}$ such that every edgelet of $K_M$ has length $1$. We make the following claim : 

\textit{Claim: } There exists a constant $C>0$, independent of $M$, such that each edgelet of $K_M\setminus Y_{M,C}$ is labeled by an edge of $H_0$.

Putting the proof of the claim off till the end of this proof, we use use the claim to show that both 
$\mathcal{F}_{\phi}, \mathcal{F}_{\psi}\sqsubset \mathcal{F}_M$. The above claim along with the fact that 
the number of natural vertices of $K_M$ is uniformly bounded above (depending only on rank of $\F$), implies that the graph 
$Y_{M,C}$ has uniformly bounded number of edgelets, independent of $M$.
 Also note that the set of edgelet labels (which are defined to be 
the edges of $H$) is finite and $K_M$ has uniformly bounded rank. Hence after passing to a subsequence, we may assume that there exists a homeomorphism 
$h_{M_i, M_j}:(K_{M_i},Y_{M_i,C}) \to (K_{M_j}, Y_{M_j,C})$ such that the restriction of $h_{M_i,M_j}$ to $Y_{M_i,C}$ maps edgelets to edgelets and 
preserves the labels. Now, since $M_i\to\infty$ the exponents of $\phi$ and $\phi$ appearing in $\xi_{M_i}$ diverges to $\infty$. This means that the expansion factor 
of $\xi_{M_i}\to\infty$ (this is visible in the proof of pingpong lemma \ref{pingpong}). Hence the edge-length of the projection of any natural edge of $K_{M_i}$ to $H$ diverges to $\infty$. 
Thus after passing to a subsequence and enlarging $Y_{M_i,C}$ if necessary, we may assume that the edgelet length of 
each component of $K_{M_i}\setminus Y_{M_i,C}$ goes to infinity with $M_i$. 

Consider a generic leaf $\gamma^+_\psi$ of the 
attracting lamination $\Lambda^+_\psi$. Since $\Lambda^+_\psi$ is not carried by $\mathcal{F}$, by our assumption (1) of relative independence \ref{relind}, we may choose 
a finite subpath $\sigma\subset\gamma_\psi$ such that $\sigma$ defines an attracting neighborhood of $\Lambda^+_\psi$ and 
$\sigma$ begins and ends with edges in $H\setminus H_0$. By using the uniformity of attracting neighborhoods from the pingpong lemma (\ref{pingpong} conclusion 7,8) we know that $\gamma^+_M$ belongs to this neighborhood for sufficiently large $M$. 
This means $\sigma \subset \gamma^+_M$ for sufficiently large $M$. Lifting everything to the stallings graph, we can see that the lift of $\sigma$ is a path contained in 
$K_M$ and the first and last edgelets of this lift is in $Y_{M,C}$ and the edgelet length of this lift is independent of $M$ (since $\gamma^+_\psi$  is not carried by $\mathcal{F}$).
For sufficiently large $M$, the edgelet length of each component of $K_{M}\setminus Y_{M,C}$ is greater than $L$, and hence 
the lift of $\sigma$ to $K_M$ must be contained in $Y_M$. Since $Y_{M,C}$ is independent of $M$, each finite subpath of 
$\gamma^+_\psi$ lifts to $Y_{M,C}$ and hence $\gamma^+_\psi$ lifts to $Y_{M,C}$. By a symmetric argument we can show that 
$\gamma^-_\phi$ lifts to $Y_{M,C}$. This implies that both $\Lambda^+_\psi$ and $\Lambda^-_\phi$ are carried by 
$[F_M]$, and hence $\Lambda^+_\psi$ and $\Lambda^-_\phi$ are carried by $\mathcal{F}$, which contradicts our standing hypothesis 
(assumption (1) in \ref{relind}). This completes the proof of the fact that the $\xi_M$'s are fully irreducible 
relative to $\mathcal{F}$ for all sufficiently large $M$.

\textit{proof of claim: } Suppose our claim is false. 
There exists a subsequence $(M_i)$ with $M_i\rightarrow \infty$ such that for each $i\geq 1$ there exists an edgelet $e_i$ of $K_M$ which projects to an edge of $H\setminus H_0$ and 
$e_i$ lies outside of $Y_{M,C}$. By taking $C$ sufficiently large, we may assume that $e_i$ is a central edgelet of an edgelet path $\epsilon_i$ of length greater than 
$2i+1$ in $K_M$ and that $\epsilon_i$ does not contain any natural vertices of $K_M$. Passing to a subsequence of $(M_i)$ we may 
assume that the image $p_{M_i}(e_i)$ in $H$ is independent of $i$. Continuing inductively by passing to further subsequences at each step
we get a subsequence of $(M_i)$ such that the image of the central $3, 5, 7, ....., 2i+1$ segment of $\epsilon_j$ in $H$ is constant independent of $j\geq i$.
Observe that each of these central segments of $\epsilon_j$ that we have constructed is in $K_M$ but lies outside of 
$Y_{M_i,C}$ and projects to an edge-path in $H\setminus H_0$ (hence lies outside $\widehat{K}_M$). This implies that both $\gamma^+_{M_i}$ and 
$\gamma^-_{M_i}$ cross $\epsilon_i$. This implies that the nested union of the projection of central segments of the $\epsilon_i$'s in $H$ is a line $\gamma$ that is a 
weak limit of some subsequence of $(\gamma^+_{M_i})_{i\geq1}$ and also a weak limit of some subsequence of 
$(\gamma^-_{M_i})_{i\geq1}$. But since this line crosses edges of $H\setminus H_0$ it is not supported by 
$\mathcal{F} = [H_0]$. The following sublemma now completes the proof:\\

 \textbf{sublemma: } If $\gamma$ is a weak limit of some subsequence of $(\gamma^+_{M_i})$ and it also a weak limit of some subsequence of 
 $(\gamma^-_{M_i})$, then $\gamma$ must be carried by $\mathcal{F}$.
 
The proof of this sublemma follows verbatim as the proof of item (9) within the proof of Theorem I in \cite{HM-13d}.
	
 \end{proof}

 We end this section with a corollary that is a direct application of the above theorem. This corollary gives us a way to construct relatively irreducible free subgroups in $\out$.

 \begin{corollary} \label{rfi}
 Given a free factor system $\mathcal{F}$ with co-edge number $\geq 2$, and given $\phi, \psi \in \out$ , if $\phi, \psi$ are fully irreducible relative 
 to $\mathcal{F}$, with corresponding invariant lamination pairs $\Lambda^\pm_\phi, \Lambda^\pm_\psi$ (as in the equivalence condition \ref{relind})such that the pair $\{\Lambda^+_\phi , \Lambda^-_\phi\}$ 
 is disjoint from the pair 
 $\{\Lambda^+_\psi, \Lambda^-_\psi\}$, then there exists an integer $M \geq 1$ such that for any $m,n \geq M$ the group $\langle \phi^m, \psi^n\rangle$ is a free group 
 of rank 2 and every element of this group is fully irreducible relative to $\mathcal{F}$. 
 Moreover, 
 \begin{enumerate}
  \item if both $\phi, \psi$ are nongeometric above $\mathcal{F}$, then every element of this free 
 group is also nongeometric above $\mathcal{F}$.
 \item if both $\phi,\psi$ are geometric above $\mathcal{F}$ and the geometric laminations $\Lambda^\pm_\phi$ and 
 $\Lambda^\pm_\psi$ come from the same surface, then every element of the free group is geometric, fully-irreducible 
 above $\mathcal{F}$ and they all have the same unique closed indivisible Nielsen path as $\phi$ and $\psi$.
 
 \end{enumerate}

 \end{corollary}
 
 \begin{proof}
  The first item in the corollary follows directly from Theorem A. Item (2) is an application of the work of Farb and Mosher 
  in \cite{FaM-02} (Theorem 1.4) where they construct Schottky subgroups in mapping class groups of surfaces. A detailed write up 
  for (2) can be found in proof of Proposition 4.7 in \cite[Page 36, Case 2]{HM-15}.
 \end{proof}

 \section{Applications}
 \label{section4}
 In this section we look at some very interesting geometric properties of for the free subgroups we have constructed in 
 Corollary \ref{rfi}. We show that the free-by-free extension groups obtained by using the free subgroups constructed 
 in Corollary \ref{rfi} are (strongly) relatively hyperbolic. As far as the knowledge of the author goes, this is the 
 first time in the theory of $\out$ that such examples have been constructed. 
 
 The short exact sequence $$1\to \F \to \aut \to \out \to 1 $$ induces a short exact sequence 
 $$ 1\to \F \to \Gamma \to \mathcal{H} \to 1$$  for any subgroup $\mathcal{H}<\out$. The extension group $\Gamma$ 
 is a subgroup of $\aut$. For our purposes any short exact sequence of groups we talk about is obtained in this manner.

 In this subsection we give sufficient conditions when the extension group $\Gamma$ in the short exact sequence 
 is strongly  relatively hyperbolic relative to some collection of subgroups 
 of $\Gamma$ for some very interesting subgroups $\mathcal{H}\in\out$. The restrictions that we impose on 
 $\mathcal{H} $ are very natural in the sense that we generalize two very interesting results  \cite[Theorem 5.1]{BFH-97} and 
 \cite[Theorem 5.2]{BFH-97}. For simplicity we shall split the proof into cases so that cumbersome notations can be avoided 
 as much as possible. To make it easy for the reader, we also give a list of all notations used under a single heading and have tried to make it 
 as intuitive as possible.

  \textbf{Coned-off Cayley graph :}
  Given a group $G$ and a collection of subgroups $H_\alpha < G$, the coned-off Cayley graph of $G$ or 
  the \textbf{electric space} of $G$ relative to the collection $\{H_\alpha\}$ is a metric space which consists of the 
  Cayley graph of $G$ and a collection of vertices $v_\alpha$ (one for each $H_\alpha$) such that each point of 
  $H_\alpha$ is joined to (or coned-off at) $v_\alpha$ by an edge of length $1/2$. The resulting metric space is 
  denoted by $(\widehat{G}, {|\cdot|}_{el})$.

  A group $G$ is said to be (weakly) relatively hyperbolic relative to the collection of subgroups 
  $\{H_\alpha\}$ if $\widehat{G}$ is a $\delta-$hyperbolic metric space, in the sense of Gromov.
  $G$ is said to be strongly hyperbolic relative to the collection $\{H_\alpha\}$ if the coned-off 
  space $\widehat{G}$ is weakly hyperbolic relative to $\{H_\alpha\}$ and it satisfies the 
  \textit{bounded coset penetration} property (see \cite{Fa-98}). 
  But 
this bounded coset penetration property is a very hard condition to check for random groups $G$. 
However if the group $G$ is hyperbolic and the collection of subgroups $\{H_\alpha\}$  is 
\textit{mutually malnormal} and \textit{quasiconvex} then $\widehat{H}$ is strongly relatively hyperbolic. 
 
 We say that a collection of subgroups $\{H_\alpha\}$ of $\F$ is a \textit{mutually malnormal collection} of subgroups if 
each $H_\alpha$ is a malnormal subgroup i.e. $w^{-1}H_\alpha w \cap H_\alpha = \{e\}$ for all $w\notin H_\alpha$, and 
 for all $\alpha\neq\beta$ the intersection $w^{-1}H_\alpha w \cap H_\beta= \{e\}$ for all $w\in \F$.

 The main tool that we will be using to prove strong relative hyperbolicity is a generalization of the Bestvina-Feighn 
 Annuli Flare Condition \cite{BF-92}. It is known as the \textit{cone bounded hallways strictly flare condition} 
 and is due to Mj-Reeves.

 Before we state the result from Mj-Reeves we need to make a definition.
 \begin{definition}\label{mappingtorus}
   Let $H_\alpha<\F$ be malnormal and $\Phi\in\aut$ be such that $\Phi(H_\alpha)=H_\alpha$. Then define the \textbf{mapping torus of}   
   $H_\alpha$ to be 
   $$\langle\langle H_\alpha, t | t^{-1}w t = \Phi(w) , \forall w\in H_\alpha \rangle  \rangle < \F\rtimes_\Phi \mathbb{Z}$$
  where the symbol $t$ comes from the definition of mapping torus $\F\rtimes_\Phi \mathbb{Z}$.
     
 \end{definition}

Now let $\mathcal{H}<\out$  be a free group of rank 2 or an infinite cyclic subgroup and $H_\alpha<\F$ be malnormal and suppose that the 
conjugacy class of $H_\alpha$ is invariant under $\mathcal{H}$. Then one can always find a lift 
$\widetilde{\mathcal{H}}\in\aut$ of $\mathcal{H}$ such that for every $\Phi\in\widetilde{\mathcal{H}}$ we have 
$\Phi(H_\alpha)=H_\alpha$. We say that the lift \textit{preserves} $H_\alpha$ and hence one can define the semi-direct product 
$H_\alpha\rtimes \widetilde{\mathcal{H}}$ as a subgroup of $\F\rtimes\mathcal{H}$.

For the convenience of the reader, we recall some of the definitions from the Mj-Reeves paper and put it in context to the work being done here.
 
 \begin{enumerate}
  \label{discussion}
  \item \textbf{Induced tree of coned-off spaces: } 
  $\Gamma$ can be seen as a tree of hyperbolic metric spaces, where the tree structure, $T$, comes from the Cayley graph of the quotient group $\mathcal{H}$.
  The edge groups and vertex groups for this tree of spaces are identified with cosets of $\F$. The maps from the edge group  to the vertex groups 
  are quasi-isometries in this case. We electrocute each vertex and edge space by coning-off the copies of $H_\alpha$ in them. The resultant is that 
  $\Gamma$ now becomes a tree of strongly relatively hyperbolic metric spaces. 
  
  We electrocute the cosets of $H_\alpha$ in $\Gamma$ and obtain the \emph{induced tree of coned-off spaces}, denoted by $\widehat{\Gamma}$.
  The tree structure here again comes from $T$ as explained above with the vertex and edge spaces now being strongly relatively hyperbolic spaces.

  \item \textbf{Cone-locus} (from \cite{MjR-08}): The cone locus of $\hat{\Gamma}$, induced tree of coned-off spaces, is the forest whose vertex set consists of the cone-points of the 
vertex spaces of $\hat{\Gamma}$ and whose edge set consists of the cone-points in the edge spaces of $\hat{\Gamma}$. The incidence relations 
of the cone locus is dictated by the incidence relations in $T$.

Connected components of cone-locus can be identified with subtrees of $T$. 
Each connected component of the cone-locus is called a \textit{maximal cone-subtree}. 
The collection of maximal cone-subtrees is denoted by $\mathcal{T}$ and each element of 
$\mathcal{T}$ is denoted by $T_\alpha$. Each $T_\alpha$ gives rise to a tree $T_\alpha$ of horosphere-like subsets depending on 
which cone-points arise as vertices and edges of $T_\alpha$. The metric space that $T_\alpha$ gives rise 
to is denoted by $C_\alpha$ and is called \textit{maximal cone-subtree of horosphere-like
spaces}. The collection of $C_\alpha$ is denoted by $\mathcal{C}$.

Notice that from the choices involved in Lemma \ref{hallway},
we have an induced short exact sequence \[1\to H_\alpha \to\Gamma_\alpha \to \mathcal{H} \to 1\] 
In our context, the collection $\mathcal{C}$ corresponds to the collection of 
cosets of $\Gamma_\alpha= H_\alpha\rtimes \widehat{\mathcal{H}}_\alpha$ as a subgroup of $\Gamma$, for each $\alpha$. Note that the partially electrocuted metric space, $\hat{\Gamma}$,
can be viewed as electrocuting cosets of $H_\alpha$ in $\Gamma_\alpha$ across all 
cosets of $\Gamma_\alpha$ in $\Gamma$, for each $\alpha$. The maximal cone-subtrees, $T_\alpha$'s, are obtained from
this electrocution of $\Gamma_\alpha$; each coset of $\Gamma_\alpha$, after electrocution of 
cosets of $H_\alpha$ inside $\Gamma_\alpha$, gives us a maximal cone-subtree, i.e. equivalently the $C_\alpha$'s are electrocuted to $T_\alpha$'s in the first 
step of electrocution.

  \item \textbf{Hallway}(from \cite{BF-92}): A disk $f: [-m, m]\times I\to \hat{\Gamma}$ is a hallway 
  of length $2m$ if it satisfies the following conditions:
  \begin{enumerate}
   \item $f^{-1}(\cup \hat{\Gamma}_v: v\in T) = \{-m, ...., m\}\times I$.
   \item $f$ maps $i\times I$ to a geodesic in $\hat{\Gamma}_v$ for some vertex space.
   \item $f$ is transverse, relative to condition (1) to $\cup \hat{\Gamma}_e$.
  \end{enumerate}
  
  Recall that in our case, the vertex spaces being considered above are just copies of $\hat{\F}$ with the electrocuted metric (obtained 
  from $\F$ by coning-off the collection of subgroups $F^i$).

\item \textbf{Thin hallway}: A hallway is $\delta-$thin if $d(f(i,t), f(i+1,t)) \leq \delta$ for 
all $i, t$.
\item A hallway is $\lambda-$hyperbolic if 
$$\lambda l(f(\{0\}\times I)) \leq \text{max}\{l(f(\{-m\}\times I)), l(f(\{m\}\times I))\}$$
 
 \item \textbf{Essential hallway}: A hallway is essential if the edge path in $T$ resulting from projecting $\hat{\Gamma}$ onto $T$
 does not backtrack (and hence is a geodesic segment in the tree $T$).
 \item \textbf{Cone-bounded hallway} (from \cite[Definition 3.4]{MjR-08}): An essential hallway of length $2m$ is cone-bounded if 
 $f(i\times \partial I)$ lies in the cone-locus for $i=\{-m, ...., m\}$. 
 
 Recall that in our case, the connected components of the cone-locus are $T_\alpha$'s which are the cosets of 
 $\Gamma_\alpha$ (post electrocuting the cosets of $H_\alpha$ in $\Gamma_\alpha$) inside $\Gamma$.

 \item \textbf{Hallways flare condition} (from \cite{BF-92}, \cite{MjR-08}): 
 The induced tree of coned-off spaces, $\hat{\Gamma}$, is said to 
 satisfy the hallways flare condition if there exists $\lambda> 1, m\geq 1$ such that 
 for all $\delta$ there is some constant $C(\delta)$ such that any $\delta-$thin essential 
 hallway of length $2m$ and girth at least $C(\delta)$ is $\lambda-$hyperbolic.

 In our context, Proposition \ref{strictflare} shows the hallways flare condition is satisfied when the quotient group is infinite cyclic and 
 Proposition \ref{34} shows the hallways flare condition is satisfied when the quotient group is free, for $\hat{\Gamma}$ (with $\lambda= 2$), 
 since $\hat{\Gamma}$ is obtained from $\Gamma$ by electrocuting cosets of $H_\alpha$, for each $\alpha$.
 Thus $\hat{\Gamma}$ is a hyperbolic metric space by using the Bestvina-Feighn combination theorem and $\hat{\Gamma}$ is weakly hyperbolic 
 relative to the collection $\mathcal{T}$ (\cite[Lemma 3.8]{MjR-08}).
 
 Once the hyperbolicity of $\hat{\Gamma}$ is established, we can proceed to the second stage of 
 electrocution, where we electrocute the the maximal cone-subtrees, $T_\alpha$'s. The resulting space is quasi-isometric to electrocuting 
 the $C_\alpha$'s inside $\Gamma$ to a cone-point directly (see proof of \cite[Theorem 4.1]{MjR-08}) and thus this step shows that $\Gamma$ 
 is weakly hyperbolic relative to the collection of spaces $C_\alpha$ (equivalently, relative to the collection of mapping tori subgroups
 $\Gamma_\alpha$).
 
 \item \textbf{Cone-bounded hallways strictly flare condition} (from \cite[Definition 3.6]{MjR-08}): 
 The induced tree of coned-off spaces $\hat{\Gamma}$, is said to satisfy the cone-bounded hallways strictly flare condition if there exists 
 $\lambda>1, m\geq 1$ such that any cone-bounded hallway of length $2m$ is $\lambda-$hyperbolic.

 In our case, this condition is also verified in Proposition \ref{strictflare} for the infinite cyclic case and Proposition \ref{34} for the free group case, since each connected component of the 
 cone-locus is just $\Gamma_\alpha$ with the 
 electrocuted metric obtained by coning-off the subgroup $H_\alpha$ and it's cosets (in $\Gamma_\alpha$) and this electrocuted $\Gamma_\alpha$ can be viewed as a subspace  of 
 $\hat{\Gamma}$ (for which the flaring condition holds). 
 
 \end{enumerate}

 The main theorem of \cite[Theorem 4.6]{MjR-08} is a very general and powerful result and can be stated 
 as the following lemma, which is a special case of that theorem, and suffices for our purposes. 
 \begin{lemma}[Theorem 4.6, MjR-08]\label{hallway}
 Let $\mathcal{H} < \out$ be infinite cyclic or a free group of rank 2.
  Suppose  $\{H_\alpha\}$ is a mutually malnormal collection of quasiconvex subgroups of $\F$ such that the conjugacy classes 
  of each $H_\alpha$ is invariant under $\mathcal{H}$. If the hallways flare condition 
  is satisfied for the induced tree of coned-off spaces defined by this collection
  and the cone-bounded hallways strictly flare condition is also satisfied 
  then the extension group $\Gamma$ defined by the 
  short exact sequence $$1 \to \F \to \Gamma \to \mathcal{H} \to 1 $$ is strongly  hyperbolic relative to the 
  the collection of subgroups $\{H_\alpha\rtimes \widetilde{\mathcal{H}}_\alpha\}$ 
  where $\widetilde{\mathcal{H}}_\alpha$ is a lift that preserves $H_\alpha$.
  
 \end{lemma}
\begin{proof}
 
 We have the following short exact sequence where the quotient group \[1\to\F\to\Gamma\to\mathcal{H}\to1\] is either a free group or an infinite 
 cyclic group. The Cayley graph of the quotient group is thus a tree, call it $T$, which enables us to view $\Gamma$ as a tree, $T$, of metric spaces.
 The vertex groups and edge groups of this tree of metric spaces are identified with cosets of $\F$ in $\Gamma$. The maps between the edge space and the two vertex spaces 
 (which are the initial and terminal vertices of the edge in consideration) are in fact quasi-isometries in this case. Since each $H_\alpha$ is 
 preserved up to conjugacy, it follows immediately that the \emph{q.i. embedded condition, strictly type preserving condition} and the \emph{q.i. preserving 
 electrocution condition} are all satisfied.

 By the condition imposed on the collection of subgroups $\{H_\alpha\}$ we know that each vertex space and each edge space is (strongly) hyperbolic 
 relative to this collection. Hence $\Gamma$ can be viewed as a tree, $T$, of strongly relatively hyperbolic spaces. Denote this tree of strongly relatively hyperbolic 
 spaces by $\Gamma$. In the Mj-Reeves theorem \cite[Theorem 4.6]{MjR-08}, $\widehat{\Gamma}$ is the \emph{induced tree of coned-off} 
 spaces which is obtained by electrocuting cosets of $H_\alpha$ in $\Gamma$. The collection of maximal parabolic subgroups in this case corresponds to the collection of subgroups $\{H_\alpha\rtimes \widetilde{\mathcal{H}}_\alpha\}$. 
 Hence if the hallways flare condition and the cone-bounded strictly flares condition is satisfied for $\widehat{\Gamma}$, the conclusion follows.
 \end{proof}

 It is worth pointing out that one could also use a combination theorem of Gautero \cite{Gau-16} to prove what we intend to 
 do with the Mj-Reeves combination theorem for the cyclic case. Infact, our proof here would fit in exactly into the framework of Gautero's 
 work and give us Theorems \ref{relhyp2} and \ref{geom}. However for the more general case when the quotient group 
 $\mathcal{H}$ in the short exact sequence $1\to\F\to\Gamma\to\mathcal{H}\to 1$ is a free group, it is perhaps much harder to deduce 
 relative hyperbolicity Gautero's work.
  
   \subsection{Relative hyperbolic extensions: nongeometric case}

   Recall that an element $\phi\in\out$ is fully irreducible and nongeometric above a multi-edge extension
   free factor system $\mathcal{F} = \{[F^1], [F^2], .... ,[F^k]\}$ if and only if 
   there exists a dual lamination pair
   $\Lambda^\pm_\phi$ such that the free factor support $\mathcal{F}_{supp}(\mathcal{F}, \Lambda^\pm) = \{[\F]\}$
   and the nonattracting subgroup system $\mathcal{A}_{na}(\Lambda^\pm_\phi)=\mathcal{F}$. In context of the strong combination theorem 
   \ref{hallway}, the collection $\{F^i\}$ is the collection $\{H_\alpha\}$ in the statement of the lemma.

   Under this hypothesis we get a strong relative hyperbolic extension $\widehat{\F}$ by performing 
   electrocution of the collection of subgroups $\{F^i\}$ since Handel and Mosher 
   prove that the nonattracting system in \cite{HM-13c} is malnormal system and in this case since we have 
   finitely generated subgroups of free groups, $F^i$ is quasiconvex. This strong relative 
   hyperbolic space is what induces the tree of strong relative hyperbolic spaces that we make use of 
   to apply the Mj-Reeves strong combination theorem.

  For simplicity of the proofs we make some assumptions and explain the connection with 
  Mj-Reeves work.

  \textbf{Notations: } For convenience of the reader we introduce and explain a few notations
 that are used in the proof.

 \begin{enumerate}
  \item $f:G\to G$ denotes a improved relative train track map for $\phi$ which is fully irreducible and nongeometric 
  above $\mathcal{F}$ and there is a filtration element $G_{r-1}$ of 
  $G$ such that $G_{r-1}$ realizes the free factor system $\mathcal{F}$. 
  \item $H_r$ is the EG strata sitting above $G_{r-1}$ that is associated to the attracting lamination $\Lambda^+_\phi$.
  \item Given a conjugacy class $\alpha$ in $\F$, we shall also denote the circuit representing $\alpha$ 
  in $G$ by $\alpha$ and work with the  length ${|\alpha|}_{H_r}$ which is described below. 
  \item ${|\alpha|}_{H_r}$ denotes the length of a path $\alpha \subset G$ with respect to $H_r$.  More precisely, 
  ${|\alpha|}_{H_r}$ is calculated by counting only edges of $H_r$.  
  \item For any given conjugacy class $\alpha$ in $\F$, ${||\alpha||}$ denotes the length of shortest representative of 
  the conjugacy class $\alpha$ and ${||\alpha||}_{el}$ will denote the length of the same 
  representative in the coned-off space $\widehat{\F}$, which is obtained from $\F$ by coning-off with respect to 
  the collection of subgroups $\{F^i\}$. Note that since $\mathcal{F}$ is a malnormal subgroup system, $\F$ is (strongly) relatively 
  hyperbolic with respect to the collection of subgroups $\{F^i\}$. 
  In what is to follow ${|\cdot|}_{el}$ will denote the electrocuted metric for this (strongly) relatively hyperbolic group. This is the electrocuted 
  metric being used in the statement of Conjugacy flaring \ref{conjflare}, strictly flaring \ref{strictflare} and the 3-of-4 strech 
  lemma \ref{34}.
  
  \item $f:G'\to G'$ denotes a improved relative train track map for  $\phi^{-1}$ which is fully irreducible and nongeometric 
  above $\mathcal{F}$ and there is a filtration element $G'_{s-1}$ of 
  $G'$ such that $G'_{s-1}$ realizes the free factor system $\mathcal{F}$.
  $H'_s$ is the EG strata sitting above $G'_{s-1}$ that is associated to the attracting lamination $\Lambda^-_\phi$.
 \end{enumerate}

  \textbf{Standing Assumption:} Continuing with the notations we have fixed above we make a list for the standing assumptions for 
  this section. For the rest of this section, unless otherwise mentioned, we will 
 assume that 
 \begin{enumerate}\label{SA}
 \item $\phi\in\out$ is  fully irreducible relative to an invariant free factor system 
 $\mathcal{F}$  and $\mathcal{F}\sqsubset [\F]$ 
 is a multi-edge extension and $\phi$ is nongeometric above $\mathcal{F}$.
 \item $\Lambda^+_\phi, \Lambda^-_\phi$ will denote the $\phi-$invariant dual lamination pair which sits above $G_{r-1}$ 
 such that $\mathcal{A}_{na}(\Lambda^\pm)=\mathcal{F}$ and $\mathcal{F}_{supp}(\mathcal{F},\Lambda^\pm)=[\F]$ (i.e. 
 $\Lambda^\pm$ fill relative to $\mathcal{F}$). 
 This is the equivalent to the above assumption by the work of Handel-Mosher \cite{HM-13d} (see \ref{firelf}).
 \item $\Phi\in\aut$ is a lift of $\phi$ such that $\Phi(F^i)={g_i}^{-1}F^ig_i$.
 \item We perform an electrocution of $\Gamma$ (the extension group in $1\to\F\to\Gamma\to\langle \phi \rangle\to 1$) 
 with respect to the collection of subgroups $\{F^i\}$. 
  We denote the resulting  metric space by $\widehat{\Gamma}$. 
  In the context of the Mj-Reeves work this is the 
 \textit{partially electrocuted space} they discuss in section 2.2 of \cite{MjR-08}. Whenever we use this electrocuted metric, 
 we shall spell it out clearly so as not to cause confusion with the electrocuted metric from $\widehat{\F}$. It shall be clear 
 from the context.
 
  \end{enumerate}

 Let $F^1*F^2 =\F$ and $k_1, k_2\in F^2$ be distinct elements. Then, $k_1, k_2$ cannot belong to the 
  same coset of $F^1$, since that would imply $k_2^{-1}k_1\in F^1$.  This shows that when $w$ 
  is written in terms of generators of $F^1$ and $F^2$, the geodesic path representing 
  the word $w$ in the coned off Cayley graph penetrates a coset for each appearance of 
  a generator of $F^2$ in representation $w$. In other words, the number of times
  generators of $F^2$ appear in the reduced form of $w$ track the length of the geodesic 
  representing $w$ in the coned off Cayley graph, up to an uniformly bounded error .

 For the proof we will follow the work of Bestvina-Feighn-Handel in \cite{BFH-97} very closely and generalize most of the 
 results in that paper to the setup that we have here. We begin with a small proposition that indicates why it is natural to cone-off 
 representatives of the nonattracting subgroup system and expecting flaring to occur in the resulting electrocuted metric space.

 \begin{proposition}\label{relflare}
  Let $\phi\in\out$  and $\mathcal{F}$ be a $\phi-$invariant free factor system such that $\mathcal{F}\sqsubset\{[\F]\}$  is a multi-edge extension
  and $\phi$ is fully irreducible relative to $\mathcal{F}$ and nongeometric above $\mathcal{F}$. 
  Then any conjugacy class not carried by $\mathcal{F}$ grows exponentially under iterates of $\phi$.
  
 \end{proposition}

 \begin{proof}
  Let $f:G\to G$ be a improved relative train track map for $\phi$ with a core filtration element $G_r$ such that $G_r$ realizes the free factor system $\mathcal{F}$. 
  Then for any conjugacy class not carried by $\mathcal{F}$, its realization in $G$ is not entirely contained in $G_r$. Let $\sigma$ be a circuit in $G$ representing any such 
  conjugacy class.

  Our hypothesis implies that  the strata $H_{r}$ is an EG strata with an associated attracting lamination 
  $\Lambda^+_\phi$ such that $\mathcal{F}_{supp}(\mathcal{F}, \Lambda^+_\phi)=[\F]$ and $\mathcal{A}_{na}(\Lambda^+_\phi)=\mathcal{F}$.
    This gives us that $\sigma$ is not carried by $\mathcal{A}_{na}(\Lambda^+_\phi)$ and hence attracted to $\Lambda^+_\phi$, which means that 
    $\sigma$ must grow exponentially in terms of edges of $H_{r}$. 
   
 \end{proof}

 \textbf{Remark:} Note that the proof of the above proposition implies that any conjugacy class which is carried by 
 $\mathcal{F}$ grows at most 
 polynomially under iteration of $\phi$, when it's length is \textit{measured by only counting  legal segments in $H_r$}. This 
 motivates the definition of \textit{legality} of circuits which follows after this. The essential idea here is that 
 circuits with sufficient ``legality'' have  exponential growth when measured with respect to edges of $H_r$.

 Observe that this proposition essentially is an indication that flaring condition is satisfied.
 However, the uniformity of exponents are not clear from this. For that we will have to work 
 more and use some delicate arguments using legality of circuits.

 \textbf{Critical Constant: }
   Let $f:G\to G$ be a improved relative train track representative for some exponentially growing  
  $\phi\in\out$ with $H_r$ being an exponentially growing strata with associated Perron-Frobenius eigenvalue $\lambda$. 
  If $BCC(f)$ denotes the bounded cancellation constant for $f$, then the number $\frac{2BCC(f)}{\lambda-1}$ is called 
  the \textit{critical constant} for $f$. It can be easily seen that for every number $C>0$ that exceeds the 
  critical constant, there is some $\mu>0$ such that if $\alpha\cdot\beta\cdot\gamma$ is a concatenation of r-legal paths where 
   $\beta $ is some r-legal segment of length $\geq C$, then the r-legal leaf segment of 
   $f^k_\#(\alpha\cdot\beta\cdot\gamma)$ corresponding to $\beta$ has length $\geq \mu\lambda^k{|\beta|}_{H_r}$.
  To summarize, if we have a path in $G$ which has some r-legal ``central'' subsegment of length greater than the
    critical constant, then this segment is protected by the bounded cancellation lemma and under iteration, length of this segment grows exponentially.
  
 Recall that for any path $\alpha$ in $G$ the notation ${|\alpha|}_{H_r}$ denotes the $H_r$-edge length of $\alpha$.
  From Bestvina-Feighn-Handel train track theory, we know that there are no \emph{closed} indivisible Nielsen paths of height $r$.
 Also recall from  train track theory that unique indivisible Nielsen path of height $r$ (if it exists)
 has exactly one illegal turn in $H_r$ and hence does not occur as a subpath of any generic leaf of $\Lambda^+_\phi$.

  Fix some constant $C$ greater than the critical constant for $f$. The following definition is due to 
  \cite[page 236]{BFH-97}
  \begin{definition}\label{leg}
  For any circuit $\alpha$ in $G$, the $H_{r}$-legality of $\alpha$ is defined as the ratio 
  $${LEG}_{H_r}(\alpha):= \frac{\text{sum of lengths of generic leaf segments of $\Lambda^+_\phi$ in $\alpha\cap H_r$ of length} \geq C}{{|\alpha|}_{H_r}}$$
 if ${|\alpha|}_{H_r}\neq 0$. Otherwise, if ${|\alpha|}_{H_r}=0$, define ${LEG}_{H_r}(\alpha)=0$.
 
 \end{definition}
Note that this definition makes sense because there is no noncontractible strata above $H_r$ under our assumption that $\phi$ is fully irreducible 
and nongeometric above $\mathcal{F}$.
 
The next lemma shows that under our assumptions, given any conjugacy class $\alpha$ for all sufficiently large $m$, at least one of 
$\phi^m_\#(\alpha)$ and $\phi^{-m}_\#(\alpha)$ gathers enough legality to exceed the length ${|\alpha|}_{H_r}$. 
The most important output of the lemma is that this exponent $m$ can be made uniform.

\begin{lemma}\label{legality}
 Suppose $\phi$ is fully irreducible relative to $\mathcal{F}$ and satisfies the standing assumptions for this section \ref{SA}.
 Then there exists $\epsilon>0$ and some $M_2>0$ such that for every conjugacy class 
 $\alpha$ not carried by $\mathcal{F}$, either ${LEG}_{H_r}(\phi^M_\#(\alpha)) \geq \epsilon$ or 
 ${LEG}_{H'_s}(\phi^{-M}_\#(\alpha))\geq \epsilon$ for every $M\geq M_2$.
\end{lemma}

\begin{proof}
 
The proof is a direct application of the weak attraction theorem. If $\gamma^+, \gamma^-$  are 
generic leaves of $\Lambda^+, \Lambda^-$ respectively, choose a sufficiently long finite subpaths 
 $\beta\subset\gamma^+, \beta'\subset\gamma^-$ such that $\beta$ is r-legal in $G$ and 
 $\beta'$ is r-legal in $G'$. By enlarging $\beta, \beta'$ if necessary assume that 
 they define weak attracting neighborhoods of $\Lambda^+, \Lambda^-$ respectively. By 
 enlarging $\beta, \beta'$ if necessary assume that 
 $C \leq \mathsf{min}\{|\beta|, |\beta'|\}$. 
 
 Now consider any conjugacy class in $\F$ which is not carried by $\mathcal{F}= \mathcal{A}_{na}(\Lambda^\pm)$.
Applying the weak attraction theorem we can deduce that there exists some $M_3$ such that 
either $\alpha\in N(G',\beta')$ or $\phi^{M_3}_\#(\beta)\in N(G, \beta)$. If the later case happens 
then using the fact that $\phi_\#(N(G, \beta)\subset N(G,\beta))$, 
we have that $\phi^{M}_\#(\beta)\in N(G, \beta)$ for all $M>M_3$. If we have $\alpha\in N(G',\beta')$, 
then using that $\phi^{-1}_\#(N(G',\beta'))\subset N(G',\beta')$ we have that 
$\phi^{-M}_\#(\alpha)\in N(G', \beta')$ for all $M>M_3$. This shows that either $\phi^M_\#(\alpha)$ 
contains a r-legal  subsegment of length $\geq C$ in $G$ or $\phi^{-M}_\#(\alpha)$ contains a s-legal 
subsegment of length $\geq C$ in $G'$.

It remains to show that there is some $\epsilon>0$ for every conjugacy class $\alpha$ either 
${LEG}_{H_r}(\phi^M_\#(\alpha))>\epsilon$ or ${LEG}_{H'_s}(\phi^{-M}_\#(\alpha))>\epsilon$.
Suppose on the contrary that this claim is false.

Then we get a sequence of conjugacy classes $\{\alpha_i\}$ such that both 
${LEG}_{H_r}(\phi^M_\#(\alpha_i))\to 0$ and ${LEG}_{H'_s}(\phi^{-M}_\#(\alpha_i))\to 0$.
This implies that we can find segments $\beta_i \subset \alpha_i'$s such that ${|\beta_i|}_{H_r}\to\infty$ 
(respct. ${|\beta_i|}_{H'_s}\to\infty$). 
 such that neither $\phi^M_\#(\beta_i)$ nor $\phi^{-M}_\#(\beta_i)$ contain 
any central r-legal (respct. s-legal) generic leaf segment of length $\geq C$ in $H_r$ (respct. $H'_s$). 
But $\beta_i$'s (by construction) are not carried by $\mathcal{F}=\mathcal{A}_{na}(\Lambda^\pm_\phi)$ and have arbitrarily long 
 lengths  in $H_r$ and $H'_s$.
Hence, after passing to a further subsequence if necessary 
 we may assume that there is common (sufficiently long) subpath $\tau$ in $H_r$ which is crossed by all the $\beta_i$'s and 
 $\tau$ is not carried by the nonattracting subgraph that is used to construct $\mathcal{A}_{na}(\Lambda^\pm_\phi)$ .
 This implies that some weak limit of some subsequence 
 of $\beta_i$'s is a line $l$ which contains the path $\tau$. We shall derive a contradiction to this.

 Now observe that  $\phi^M_\#(l)$ (respct. $\phi^{-M}_\#(l)$) does not
 contain any subsegment of a generic leaf of $\Lambda^+_\phi$  (respc. $\Lambda^-_\phi$) of length $\geq C$ in $H_r$ (respct. $H'_s$). Using the 
 weak attraction theorem again, this implies that $l$ is carried by $\mathcal{A}_{na}(\Lambda^\pm_\phi)$ and hence contained in the nonattracting
 subgraph. But this contradicts that $l$ contains the subpath $\tau$.

 Hence there exists some $\epsilon>0$ such that 
for every conjugacy class $\alpha$, either \\ ${LEG}_{H_r}(\phi^M_\#(\alpha))>\epsilon$ 
or ${LEG}_{H'_s}(\phi^{-M}_\#(\alpha))>\epsilon$.
\end{proof}

 \begin{lemma}\label{flare}
 Suppose $\phi$ is fully irreducible relative to $\mathcal{F}$ and satisfies the standing assumptions for this section \ref{SA}.
 For every $\epsilon>0$ and $A>0$, there is $M_1$ depending only on $\epsilon, A$ such that 
 if ${LEG}_{H_{r}}(\alpha)\geq\epsilon $ for some circuit $\alpha$ then 
 $${|f^m_\#(\alpha)|}_{H_r}\geq A {|\alpha|}_{H_r} $$ for every $m>M_1$. 
\end{lemma}

\begin{proof}
 Let $\mu>0$ be as described above in the description of critical constant. Then 
 $${|f^m_\#(\alpha)|}_{H_r}\geq \mu\lambda^m |\text{\small sum of lengths of generic leaf segments of } \Lambda^+_\phi \text{ in $\alpha\cap H_r$ of length }\geq C|$$
 $\implies{|f^m_\#(\alpha)|}_{H_r}\geq \mu\lambda^m\epsilon {|\alpha|}_{H_r}$. 
 
 If we choose 
 $m$ to be sufficiently large then we can set $ A \leq \mu\lambda^m\epsilon $ and we have the desired 
 inequality.
\end{proof}

 \begin{lemma}\label{comparison}
    Suppose $\phi$ is fully irreducible relative to $\mathcal{F}$ and satisfies the standing assumptions for this section \ref{SA} 
    and $\alpha$ is any circuit in $G$ with ${|\alpha|}_{H_r}>0$. Then 
 there exists some $K>0$, independent of $\alpha$, such that $K \geq {|\alpha|}_{H_r}/{||\alpha||}_{el}\geq \frac{1}{K}$. 
 
 In particular, this inequality is true for every circuit that represents a conjugacy class which is not carried by $\mathcal{F}$.
   \end{lemma}

\begin{proof}
 Let $\tilde{G}$ denote the universal cover of the marked graph $G$. Consider $\tilde{G}$ equipped with the electrocuted metric obtained by electrocuting 
 images of all cosets of $F^i$'s, by using the marking on $\tilde{G}$, which is obtained by lifting the marking on $G$. 
 This electrocuted metric is denoted by ${|\cdot|}^{\tilde{G}}_{el}$. 
 Then any path in $\tilde{G}$ that is entirely contained in the 
 copy of a coset of some $F^i$ projects down to a path in $G$ which is carried by nonattracting subgraph $\mathcal{Z}$ that defines the 
 nonattracting subgroup system $\mathcal{A}_{na}(\Lambda^\pm_\phi)$. Note that the 
 assumption on legality of $\alpha$ ensures that we are working with a conjugacy class which is not carried by $\mathcal{A}_{na}(\Lambda^\pm_\phi)$, 
 equivalently the circuit $\alpha$ representing such a conjugacy class is not carried by $\mathcal{Z}$, hence 
 ${|\alpha|}_{H_r}\geq 1$.
 
 Let $E_i$ be some edge not carried by $\mathcal{Z}$. Then the lift $\tilde{E}_i$ is a path in $\tilde{G}$ which 
 is not entirely contained in the copy of some coset of $F^i$. Let $L_2 = \text{max}\{{|\tilde{E}_i|}^{\tilde{G}}_{el}\}$ where $E_i$ varies 
 over all edges of $G$ which are not carried by $\mathcal{Z}$. If $\alpha$ is any circuit in $G$ which is not carried 
 by the nonattracting subgroup system, then  
 \[ {|\tilde{\alpha}|}^{\tilde{G}}_{el}\leq L_2 {|\alpha|}_{H_r} 
 \implies {|\tilde{\alpha}|}^{\tilde{G}}_{el}/ {|\alpha|}_{H_r}\leq L_2\]
 
 Next suppose that $\tilde{\beta}_i$ is some geodesic path in $\tilde{G}$ (in the electrocuted metric) which connects copies of two electrocuted cosets and 
 does not intersect any copy of any electrocuted coset except at the endpoints. Note that there are only finitely many such paths upto translation
 in $\tilde{G}$. Let $L'_1=\text{max}\{{|\beta_i|}_{H_r}\}$, where $\beta_i$ is the projection of 
 $\tilde{\beta}_i$ to $G$ followed by tightening. Also consider all $\tilde{\alpha}^j_i$'s where $\tilde{\alpha}^j_i$ varies over geodesic paths   
  inside the copy of the identity coset of $F^j$ in $\tilde{G}$ representing a generator $g^j_i$ of $F^j$, under standard metric on $\tilde{G}$ 
  (i.e. we are recording the length of geodesic paths representing a generator of $F^j$ with the standard metric on $\tilde{G}$).
 Let $L''_1=\max\{{|\alpha_i^j|}_G\}$,  where $\alpha^j_i$ is the projection of $\tilde{\alpha}^j_i$ to $G$ followed by tightening (note that 
 the measurement done for $L_1''$ is in terms of standard path metric on $G$).
  
  Suppose $w\in\F$ is some cyclically reduced word not in the union of $F^i$'s. Let $\tilde{\alpha}$ be a path in $\tilde{G}$ that 
  represents the geodesic connecting the identity element to $w$, under the lift of the marking on $G$. Suppose $\tilde{\alpha}= u_1 X_1 u_2 X_2...u_s X_s$, where $X_i$'s are 
  geodesic paths in $\tilde{G}$ connecting two points in a copy of some coset via the attached cone-point and $u_i$'s are geodesic paths in 
  $\tilde{G}$ which connect copies of two electrocuted cosets and does not pass through any cone-point. Note that $u_i$'s are concatenation 
  of paths of type $\tilde{\beta}_i$'s described above.
  Under this setup we have 
  \[{|\tilde{\alpha}|}^{\tilde{G}}_{el}={|u_1|}^{\tilde{G}}_{el} + {|u_2|}^{\tilde{G}}_{el} + ...+ {|u_s|}^{\tilde{G}}_{el} + s \]
 Modify $\tilde{\alpha}$ by replacing each $X_i$ with a (minimal) concatenation of translations of paths of type $\tilde{\alpha}^j_i$ in $\tilde{G}$
 described above, and look at the 
  projection of this modified path obtained from $\tilde{\alpha}$ to $G$ followed by tightening. Denote the tightened, projected path by $\alpha$. Then 
  $\alpha$ is a path in $G$ such that after accounting for cancellations that appear in the modification of $\tilde{\alpha}$, we have
 \[{|\alpha|}_{H_r}\leq {|\tilde{\alpha}|}^{\tilde{G}}_{el}L_1 + 2sL_1
 \implies \frac{{|\alpha|}_{H_r}}{{|\tilde{\alpha}|}^{\tilde{G}}_{el}}
 \leq L_1+\frac{2sL_1}{{|\tilde{\alpha}|}^{\tilde{G}}_{el}}\leq 3L_1\]
 
 The other possibilities in the presentation of $\tilde{\alpha}$ (in terms of $u_i$'s and $X_i$'s) are handled similarly.
 Hence we have
 \[\frac{1}{3L_1}\leq  {|\tilde{\alpha}|}^{\tilde{G}}_{el}/{|\alpha|}_{H_r} \leq L_2 \addtag \]

 Now using the lift of marking map on $G$ to $\tilde{G}$ one can show by similar arguments as above that there exists some $K'>0$ such that 
 for every cyclically reduced word $w\in\F\setminus\cup F^i$ we have
 \[\frac{1}{K'}\leq {|w|}_{el}/{|\tilde{\alpha}_w|}^{\tilde{G}}_{el} \leq K' \addtag \] where $\tilde{\alpha}_w$ is the electrocuted geodesic in 
 $\tilde{G}$ connecting the image of identity element to image of $w$ under the marking map on $\tilde{G}$ and ${|w|}_{el}$ is the length of the 
 electrocuted geodesic in $\hat{\F}$ connecting identity element and $w$. Hence combining the inequalities (1) and (2) 
 above we can conclude that there 
 exists some $K>0$ such that \[ K \geq {|\alpha|}_{H_r}/{||\alpha||}_{el}\geq \frac{1}{K} \]
 \end{proof}

We now use the above lemma to prove ``conjugacy-flaring'' which is the first step towards proving the 
cone-bounded hallways strictly flare condition.
 \begin{lemma}[Conjugacy flaring]
 \label{conjflare}
 Suppose $\phi$ is fully irreducible relative to $\mathcal{F}$ and satisfies the standing assumptions for this section \ref{SA}.
  There exists some $M_0>0$ such that for every conjugacy class $\alpha$ not carried by $\mathcal{F}$, we have 
  $$3{||\alpha||}_{el} \leq \mathsf{max} \{{||\phi^M_\#(\alpha)||}_{el}, {||\phi^{-M}_\#(\alpha)||}_{el}\} $$
 for every $M\geq M_0$.
 \end{lemma}

 \begin{proof}
  Let $M_2$ be as in Lemma \ref{legality} and choose some number $D>0$ such that 
  ${|\phi^{M_2}_\#(\alpha)|}_{H_r} \geq {|\alpha|}_{H_r}/D$ and 
  ${|\phi^{-M_2}_\#(\alpha)|}_{H'_s} \geq {|\alpha|}_{H'_s}/D$. 
  By applying Lemma \ref{legality} there exists some $\epsilon >0$ such that 
  either ${LEG}_{H_r}(\phi^m_\#(\alpha)) \geq \epsilon$ or ${LEG}_{H'_s}(\phi^{-m}_\#(\alpha)) \geq \epsilon$
  for all $m\geq M_2$.
  Next choose some constant $K$ such that for every conjugacy class $\alpha$ 
  not carried by $\mathcal{F}$ we have by using Lemma \ref{comparison}, 
  $K \geq {|\alpha|}_{H_r}/{||\alpha||}_{el}\geq \frac{1}{K}$ or 
  $K\geq{|\alpha|}_{H'_s}/{||\alpha||}_{el} \geq \frac{1}{K}$ 
  depending on whether 
  ${LEG}_{H_r}(\phi^m_\#(\alpha)) \geq \epsilon$ or ${LEG}_{H'_s}(\phi^{-m}_\#(\alpha)) \geq \epsilon$.

  For concreteness assume that ${LEG}_{H_r}(\phi^m_\#(\alpha)) \geq \epsilon$. Then by applying 
  Lemma \ref{flare} with $\epsilon$ and $A= 3DK^2$ we get that there exists some $M_1$ such that for all 
  $m\geq M_0:=\mathsf{max}\{M_1, M_2 \}$
  \begin{equation}
  \begin{split}
  {||\phi^m_\#(\alpha)||}_{el} & \geq  \frac{1}{K}{|\phi^m_\#(\alpha)|}_{H_r} \\
  & \geq \frac{1}{K}3DK^2{|\phi^{M_1}_\#|}_{H_r} \\
   & \geq 3DK \frac{1}{D}{|\alpha|}_{H_r}= 3K{|\alpha|}_{H_r}\\
   & \geq 3K\frac{1}{K}{||\alpha||}_{el} = 3{||\alpha||}_{el}
  \end{split}
 \end{equation}
 
 The other part of the inequality follows from a symmetric argument.

 \end{proof}

 Finally we are ready to prove the cone-bounded hallways strictly flare condition by using conjugacy flaring
 property that we just proved. In the statement of the following proposition we choose a lift $\Phi\in\aut$ of $\phi$ and 
 representatives $\F^i$ of the conjugacy classes $[F^i]$. Then $\Phi$ takes each $F^i$ to some conjugate of 
 $F^i$.
 If we replace $\Phi$ by another lift $\Phi\circ\iota_{g}$ for some 
 $g\in \F\setminus \cup F^i$, the proposition is still true. This is easy to see, since the endpoints of the 
  geodesics corresponding to $\Phi^k(w)$ and ${(\Phi\circ \iota_g)}^k(w)$ in the coned-off space are uniformly bounded distance from each other.
 
 \begin{proposition}[Strictly flaring]
 \label{strictflare}
 Suppose $\phi$ is fully irreducible relative to $\mathcal{F}$ and satisfies the standing assumptions for this section \ref{SA}.
  There exists some $N>0$ such that for every word $w\in \F\setminus \cup F^i$ we have 
  $$2{|w|}_{el} \leq \mathsf{max}\{{|\Phi^n_\#(w)|}_{el}, {|\Phi^{-n}_\#(w)|}_{el}\}$$
  for every $n \geq N$.
 \end{proposition}

 \begin{proof} Using the fact that for any subgraph $H\subset G$ we can form a free factor system by using the 
 fundamental groups of the noncontractible components of $H$, we form a free factor system $\mathcal{K}$ by viewing 
 $H_r$ as a subgraph of $G$ and by $[K_i]$ we denote a component of $\mathcal{K}$ corresponding to each noncontractible component of $H_r$. 
   Also let $L = \mathsf{max}_{i,j}\{{|\Phi^i_\#(k_j)|}_{el}\arrowvert i=0,\pm1,\pm2,...,\pm M_0\}$ where 
  $k_j$ varies over all the basis elements for some chosen basis of $K_s$, for each component $[K_s]$ of $\mathcal{K}$
  and $M_0$ is the constant from Lemma \ref{conjflare}.  We also have the inequality ${|\Phi^{i}_\#(k_j)|}_{el}\geq 1/L$ 
  for all $i, j$ as described above. 
  
  \textbf{Case 1: } Assume $w\in\F\setminus \cup F^i$ and ${|w|}_{el} \geq L-3$.
  
  The proof is by induction. For the base case let $n = M_0$.

   If $w$ is a cyclically reduced word then conjugacy class of $w$ is not carried by $\mathcal{F}$ and 
  so by using Lemma \ref{conjflare} we have  
  $$ \mathsf{max}\{{|\Phi^n_\#(w)|}_{el}, {|\Phi^{-n}(w)|}_{el}\} \geq 3{|w|}_{el} \geq 2{|w|}_{el} $$
  If $w$ is not cyclically reduced then we can choose a basis element $k\in K_s$ (where $K_s$ is as used in the 
  description of the constant $L$ above) such that 
  $kw\in \F\setminus \cup F^i$  is a cyclically reduced word.
  Hence we get the same inequality as above, but with $w$ being substituted by $kw$.
  
  For sake of concreteness suppose that ${|\Phi^n_\#(kw)|}_{el} \geq 3{|kw|}_{el}$. Then we have 
  $3{|kw|}_{el}\leq {|\Phi^n_\#(kw)|}_{el}\leq  {|\Phi^n_\#(w)|}_{el}+{|\Phi^n_\#(k)|}_{el}\leq {|\Phi^n_\#(w)|}_{el} +L $.
  This implies that $3 + 3{|w|}_{el} - L \leq 3{|kw|}_{el} - L \leq {|\Phi^n_\#(w)|}_{el}$ 
  since ${|k|}_{el} = 1$ (as $k$ is a basis element) and there is no cancellation between $k$ and $w$. 
  Since we have  ${|w|}_{el} \geq L-3$, the above inequality then implies $2{|w|}_{el}\leq {|\Phi^n_\#(w)|}_{el}$
  and we are done with the base case for our inductive argument.
  
  Now assume that $M_0< n $ for the inductive step. First observe that from what we have proven so far,
  given any integer $s>0$ we have either ${|\Phi^{sM_0}_\#(w)|}_{el}\geq 2^{s}{|w|}_{el}$ or
  ${|\Phi^{-sM_0}_\#(w)|}_{el} \geq 2^{s}{|w|}_{el}$. Fix some positive integer $s_0$ such that 
  $2^{s_0}> 2L$. Any integer $n>s_0M_0$ can be written as $n= sM_0 + t$ where $0\leq t <M_0$ and 
  $s_0\leq s$. If ${|\Phi^{sM_0}_\#(w)|}_{el}\geq 2^{s}{|w|}_{el}$ then we can deduce 
  $$ {|\Phi^n_\#(w)|}_{el} = {|\Phi^{sM_0+t}_\#(w)|}_{el}\geq 2^s{|w|}_{el}/L \geq 2{|w|}_{el}$$ 
   Similarly 
  when ${|\Phi^{-sM_0}_\#(w)|}_{el}\geq 2^{s}{|w|}_{el}$ one proves by using  a symmetric argument that 
  ${|\Phi^{-n}_\#(w)|}_{el} \geq 2{|w|}_{el}$.
  
  \textbf{Case 2: } Assume $w\in\F\setminus \cup F^i$ and ${|w|}_{el} < L-3$.
  
  Firstly we note that $w\notin \cup F^i$ implies that $0 < {|w|}_{el}$. If $w$ is not conjugate to an element of some
  $F^i$, then the argument given in the beginning of Case 1 works here and we have  
  $\mathsf{max}\{{|\Phi^n_\#(w)|}_{el}, {|\Phi^{-n}(w)|}_{el}\} \geq 3{|w|}_{el} \geq 2{|w|}_{el} $ for all 
  $n\geq M_0$ by using Lemma \ref{conjflare}.
  
  If $w$ is conjugate to some element of $F^i$ then we can write $w=ugu^{-1}$ for some basis element
  $u\in K_j$ such that $u$ is not conjugate to any word in $F^i$ and $g\in F^i$ (since $\mathcal{A}_{na}(\Lambda^\pm_\phi)$ is 
  a malnormal subgroup system).
  
  ${|w|}_{el} < L-3$ implies that  
  $${|w|}_{el} \leq {|u|}_{el} + {|u^{-1}|}_{el} = 2 {|u|}_{el} < L-3$$ 
  Now observe that under iteration of $\Phi$, the reduced word $g$ has polynomial growth in the electrocuted metric 
  since it's conjugacy class is carried by the nonattracting subgroup system $\mathcal{F}$, whereas the word 
  $u$ grows exponentially under iteration of $\phi$. Hence we can conclude that 
  ${|\Phi^s_\#(w)|}_{el} \geq {|\Phi^s_\#(u)|}_{el}$ for all $s>0$ and thus $w$ has exponential growth in the electrocuted 
  metric. Now choose some $N_w$ such that ${|\Phi^{N_w}_\#(u)|}_{el}\geq 4 {|u|_{el}}\geq 2{|w|}_{el}$.  
  Observe that the bounded cancellation lemma tells us that $N_w$'s obtained from this subcase depend only on the conjugating word $u$ and not on 
  $g$. Hence they are only finitely many $N_w$'s.
  
  Finally we let $N$ be max of $M_0$ from Case 1 and all the $N_w$'s from case 2 
  and we have the desired conclusion.
  
  \end{proof}

  With the above Proposition \ref{strictflare} in place we have the pieces needed to apply the Mj-Reeves 
  strong combination theorem. Recall that we are working with an outer automorphism $\phi$ which is fully irreducible 
  relative to a free factor system $\mathcal{F}=\mathcal{A}_{na}(\Lambda^\pm_\phi)$, where $\Lambda^\pm_\phi$ 
  are nongeometric above $\mathcal{F}$. We have chosen some representative $F^i$ for each component $[F^i]$ of $\mathcal{F}$. 
  Then we performed a partial electrocution of the extension group $\Gamma$ $$1\to\F\to\Gamma\to\langle\phi\rangle\to 1 $$
  with respect the collection of 
  subgroups $\{F^i\}$ and denoted it by $(\widehat{\Gamma}, {|\cdot|}_{el})$. We also performed an electrocution of 
  $\F$ with respect to the collection $\{F^i\}$, denoted by $\widehat{\F}$, and since $\mathcal{A}_{na}(\Lambda^\pm_\phi)$
  is a malnormal subgroup system, $\F$ is (strongly) relatively hyperbolic with respect to the collection $\{F^i\}$. 
  The Cayley graph of the quotient group $\langle\phi\rangle$ 
  being a tree, gives us a tree of (strongly) relatively hyperbolic spaces with vertex spaces being identified 
  with cosets of $\F$. Thus we may regard the Cayley graph of $\Gamma$ as a tree of (strongly) relatively hyperbolic 
  spaces and then $\widehat{\Gamma}$ is the induced tree of coned-off spaces in the statement of the Mj-Reeves strong combination theorem \ref{hallway}. 
  Proposition \ref{strictflare} proves that the hallway flare condition and the cone-bounded hallways strictly flare condition 
  are satisfied for this induced tree of coned-off spaces (see discussion after the definition \ref{mappingtorus} items (7) and (8)).

 In light of the above discussion, we then have the following theorem by using Lemma \ref{hallway}.
 
 \begin{theorem}
 \label{relhyp2}
  Let $\phi\in\out$ and $\mathcal{F}=\{[F^1], [F^2],..., [F^k]\}$ be a $\phi-$invariant free factor system such that 
  $\mathcal{F}\sqsubset\{[\F]\}$  
  is a multi-edge extension and $\phi$ is fully irreducible relative to $\mathcal{F}$ and nongeometric above $\mathcal{F}$. 
  Then the extension group $\Gamma$ in the short exact sequence
  $$1 \to \F \to \Gamma \to \langle \phi \rangle \to 1$$ is strongly hyperbolic relative
  to the collection of subgroups $\{F^i{\rtimes_\Phi}_i \mathbb{Z}\}$, where $\Phi_i$ is a 
  chosen lift of $\phi$ such that $\Phi_i(F^i)=F^i$.
 \end{theorem}

 \begin{corollary}\cite[Theorem 5.1]{BFH-97}
  If $\phi\in\out$ is fully-irreducible and atoroidal, then the the mapping torus of $\phi$ is word hyperbolic. 
 \end{corollary}
\begin{proof}
 Apply Theorem \ref{relhyp2} with $\mathcal{F}=\emptyset$.
\end{proof}

 Now we proceed to extend this theorem to prove the main result of this work, namely, construct free-by-free relatively hyperbolic extensions. For 
 this we will first need to prove an analogous version of Lee Mosher's 3-of-4 stretch lemma. 
 The following definition is due to \cite[Definition 1.5]{BFH-97}
 \begin{definition}
  A sequence of conjugacy classes $\{\alpha_i\}$ is said to \textit{approximate} $\Lambda^+_\phi$ if 
  for any $L>0$, the ratio
  $$\frac{m(x\in S^1_i| \text{the L-nbd of x is a generic leaf segment of } \Lambda^+_\phi)}{m(S^1_i)}$$
  converges to $1$ as $i \to\infty$, where $m$ is the scaled Lebesgue measure and 
  $\tau_i:S^1_i \to G$ denotes the immersion that gives the circuit in $G$ representing 
  $\alpha_i$.
 \end{definition}

 \begin{lemma}\label{conv}
  Let $\phi, \psi$ be fully-irreducible relative to $\mathcal{F}$ such that $\mathcal{F}\sqsubset\{[\F]\}$ 
  is a multi-edge extension and suppose both are nongeometric above
  $\mathcal{F}$. 
  If $\phi$ and $\psi$ are independent relative to $\mathcal{F}$, then for any sequence 
  of conjugacy classes $\{\alpha_i\}$, the sequence cannot approximate both $\Lambda^-_\phi$ and $\Lambda^-_\psi$.
 \end{lemma}

 \begin{proof}
  Suppose that $\{\alpha_i\}$ approximates $\Lambda^-_\psi$. Since our hypothesis says that 
  $\Lambda^\pm_\psi$ is weakly attracted to $\Lambda^-_\phi$ under iterations of $\phi^{-1}$, we may choose 
  an attracting neighborhood $V^-_\phi$ of $\Lambda^-_\phi$ defined by a long generic leaf segment of $\Lambda^-_\phi$ such that 
  $\Lambda^\pm_\psi\notin V^-_\phi$. If such a leaf segment does not exist then it would imply that 
  $\Lambda^-_\phi\subset \Lambda^+_\psi$ or $\Lambda^-_\phi\subset \Lambda^-_\psi$ and in either case we would 
  violate that $\Lambda^-_\phi$ is weakly attracted to both $\Lambda^+_\psi$ and $\Lambda^-_\psi$. 
  Now notice that under this setup, $\Lambda^-_\psi\notin V^-_\phi$ and $ \Lambda^-_\psi$ is not carried by 
  $\mathcal{A}_{na}(\Lambda^\pm_\phi)$. Hence by applying the uniformity part of the weak attraction theorem, we 
  get an $M\geq 1$ such that $\phi^m_\#(\gamma^-_\psi)\in V^+_\phi$ for all $m\geq M$ for every generic leaf 
  $\gamma^-_\psi\in\Lambda^-_\psi$.  Since $V^+_\phi$ is an open set 
  we can find an $I\geq 1$ such that $\phi^m_\#(\alpha_i)\in V^+_\phi$ for all $m\geq M, i\geq I$. This 
  implies that $\{\alpha_i\}$ cannot approximate $\Lambda^-_\phi$. 
  
 \end{proof}
 
 For the next proposition we perform a partial electrocution by coning-off $\Gamma$ in the 
 short exact sequence $1\to\F\to\Gamma\to Q\to 1$, with respect to the 
 collection of subgroups $\{F^i\}$ and also perform an electrocution on $\F$ by coning-off the same collection of subgroups. 
 The resulting electric metric ${|\cdot|}_{el}$ on $\widehat{\F}$ is one one used in the statements below.

  We also choose lifts $\Phi, \Psi$ 
 of $\phi, \psi$ respectively such that there are $g_i\in\F$ and we have
 $\Phi(F^i)=\Psi(F^i)=g_i^{-1}F^i g_i$.
 
 The following theorem originates from Lee Mosher's work on mapping class groups \cite{Mos-97}.
 For the free groups case it was first shown in \cite{BFH-97} for fully irreducible nongeometric elements.
 
  \begin{proposition}[3-of-4 stretch]
 \label{34}
  Let $\phi, \psi$ be fully-irreducible relative to a $\phi-$invariant
   free factor system $\mathcal{F}=\{[F^1], [F^2],..., [F^k]\}$ such that 
  $\mathcal{F}\sqsubset\{[\F]\}$  is multi-edge extension.
  Suppose both are nongeometric above
  $\mathcal{F}$. If $\phi$ and $\psi$ are independent relative to $\mathcal{F}$ then we have the 
  following:
  \begin{enumerate}
   \item There exists some $M\geq 0$ such that 
   for any conjugacy class $\alpha$ not carried by $\mathcal{F}$,  at least three of the four numbers 
   $$ {||\phi^{n_i}_\#(\alpha)||}_{el}, {||\phi^{-n_i}_\#(\alpha)||}_{el}, 
  {||\psi^{n_i}_\#(\alpha)||}_{el}, {||\psi^{-n_i}_\#(\alpha)||}_{el}$$
  are greater an or equal to $3{||\alpha||}_{el}$, for all $n_i\geq M$.
  
  \item There exists some $N\geq 0$ such that for any word $w\in\F\setminus\cup F^i$, at least three of the four numbers 
   
   $$ {|\Phi^{n_i}_\#(w)|}_{el}, {|\Phi^{-n_i}_\#(w)|}_{el}, 
  {|\Psi^{n_i}_\#(w)|}_{el}, {|\Psi^{-n_i}_\#(w)|}_{el}$$
  are greater than $2{|w|}_{el}$, for all $n_i\geq N$.
   
  \end{enumerate}

 \end{proposition}

 \begin{proof}
 Proof of (1): 
  Suppose there does not exist any such $M_0$. We argue to a contradiction by using the weak attraction theorem. 
  By our supposition we get a sequence of conjugacy classes 
  $\alpha_i$ such that at least two of the four numbers 
  ${||\phi^{n_i}_\#(\alpha_i)||}_{el}, {||\phi^{-n_i}_\#(\alpha_i)||}_{el}$, 
  ${||\psi^{n_i}_\#(\alpha_i)||}_{el}, {||\psi^{-n_i}_\#(\alpha_i)||}_{el}$ are less than 
  $3{||\alpha_i||}_{el}$ and $n_i>i$. Proposition \ref{conjflare} tells us that at least one of 
  $\{{||\phi^{n_i}_\#(\alpha_i)||}_{el}, {||\phi^{-n_i}_\#(\alpha_i)||}_{el}\}$ is $\geq 3{||\alpha_i||}_{el}$
   and at least one of 
   
   $\{{||\psi^{n_i}_\#(\alpha_i)||}_{el}, {||\psi^{-n_i}_\#(\alpha_i)||}_{el}\}$ is $\geq 3{||\alpha_i||}_{el}$
  for all sufficiently large $i$. 
  
  For sake of concreteness suppose that 
  ${||\phi^{n_i}_\#(\alpha_i)||}_{el} \leq 3{||\alpha_i||}_{el}$ and 
  ${||\psi^{n_i}_\#(\alpha_i)||}_{el} \leq 3{||\alpha_i||}_{el}$ for all $n_i$. \hfill ($\ast$)
  
  Using Lemma \ref{conv} we know that the sequence $\{\alpha_i\}$ cannot approximate both $\Lambda^-_\phi$ and and 
  $\Lambda^-_\psi$.  For concreteness suppose that $\{\alpha_i\}$ does not approximate $\Lambda^-_\phi$.
  Then we can choose some attracting neighborhood $V^-_\phi$ of $\Lambda^-_\phi$ which is defined by some long 
  generic leaf segment and after passing to a subsequence if necessary we may assume that $\alpha_i\notin V^-_\phi$ for all 
  $i$. Also recall that our hypothesis implies that $\alpha_i$ is not carried by $\mathcal{A}_{na}(\Lambda^\pm_\phi)=\mathcal{F}$ 
  for all $i$. Hence by using the uniformity part of the weak attraction theorem, there exists some $M$ 
  such that $\phi^m_\#(\alpha_i)\in V^+_\phi$ for all $m\geq M$. Choosing $i$ to be sufficiently large 
  we may assume that $n_i \geq M$ and so by using Lemma \ref{legality} we have ${LEG}_{H_r}(\phi^{n_i}_\#(\alpha_i))\geq \epsilon$ for some $\epsilon>0$.
  By using Lemma \ref{flare} we obtain that for any $A>0$ there exists some $M_1$ such that 
  $${|\phi^m_\#(\alpha_i)|}_{H_{r}}\geq A {|\alpha_i|}_{H_{r}} $$ for every $m>M_1$. This implies that 
  for all sufficiently large $i$, $|\phi^{n_i}_\#(\alpha_i)|_{H_r} \geq A |\alpha_i|_{H_r}$. Choosing a 
  sequence $A_i\to \infty$ and after passing to a subsequence of $\{n_i\}$ we may assume that 
  $|\phi^{n_i}_\#(\alpha_i)|_{H_r} \geq A_i |\alpha_i|_{H_r}$. But this implies that the ratio 
  $|\phi^{n_i}_\#(\alpha_i)|_{H_r}/|\alpha_i|_{H_r} \to \infty$ as $i\to\infty$. This contradicts ($\ast$).
  
  Proof of (2) is similar to proof of Proposition \ref{strictflare}.
  
 \end{proof}

 Now we are ready to state the main theorem of this section, which is a generalization of 
 \cite[Theorem 5.2]{BFH-97}. Their result is obtained by taking $\mathcal{F}$ to be trivial.
 
 \begin{theorem}\label{nongeoext}
  Suppose $\phi,\psi\in\out$  and 
  $\mathcal{F}=\{[F^1], [F^2],..., [F^k]\}$ be a $\phi$ and  $\psi-$invariant free factor system such that 
  $\mathcal{F}\sqsubset\{[\F]\}$  
  is a multi-edge extension and $\phi, \psi$ are fully irreducible relative to $\mathcal{F}$ 
  and both are nongeometric above $\mathcal{F}$ and pairwise independent relative to $\mathcal{F}$.
  If $Q=\langle \phi^m, \psi^n \rangle$ denotes 
  the free group in the conclusion of corollary \ref{rfi}, then
  the extension group $\Gamma$ in the short exact sequence 
  $$ 1 \to \F \to \Gamma \to Q \to 1 $$ is strongly relatively hyperbolic with respect to the 
  collection of subgroups 
  $\{F^i \rtimes \widehat{Q_i} \}$, where  $\widehat{Q_i}$  is a lift that preserves $F^i$
  
 \end{theorem}

 \begin{proof}
  The conclusion that $\widehat{Q}$ is free group of rank 2 follows from the fact that $Q$ is a free group of rank 2.
  The cone-bounded strictly flare condition is obtained from Proposition \ref{34}. Apply Lemma \ref{hallway} to get get the 
  conclusion.
  
 \end{proof}

 \begin{corollary}\cite[Theorem 5.2]{BFH-97}
  Suppose $\phi,\psi$ are irreducible and hyperbolic outer automorphisms which do not have a common power. Then there 
  exists some $M>0$ such that for every $m, n \geq M$ the group $Q:=\langle \phi^m, \psi^n \rangle$ is a free group of 
  rank $2$ and the extension group $\F\rtimes \widetilde{Q}$ is word hyperbolic.
 \end{corollary}

 \begin{proof}
 Apply Theorem \ref{nongeoext} with $\mathcal{F}=\emptyset$.
 \end{proof}

 \subsection{Relative hyperbolic extensions: geometric case}
 In this subsection we work with  $\phi\in\out$ and a $\phi-$invariant free factor system 
 $\mathcal{F}=\{[F^1], [F^2],...., [F^k]\}\sqsubset \{[\F]\}$ which is a multi-edge extension 
 such that $\phi$ is fully irreducible relative to $\mathcal{F}$ but geometric above $\mathcal{F}$.
 This is equivalent to saying that there exists a dual  lamination pair $\Lambda^\pm_\phi$ and a
 conjugacy class $[\sigma]$ such that the following are true:
 \begin{enumerate}
  \item $\phi([\sigma])=[\sigma]$ 
  
  \item $\mathcal{F}_{supp}(\mathcal{F}, \Lambda^\pm_\phi)=\{[\F]\}$ i.e. $\Lambda^\pm_\phi$ fill relative to 
  $\mathcal{F}$.
  \item $\mathcal{A}_{na}(\Lambda^\pm_\phi)= \mathcal{F}\cup \{[\langle \sigma \rangle]\}$.
  \item $\sigma\in\F$ is primitive.
 \end{enumerate}

 The condition on $\sigma$ to be \textbf{primitive}, means $\sigma$
 is not a non-trivial power of any element of $\F$. Also recall that the nonattracting subgroup 
 system is a malnormal system. Now, in this case we slightly modify the definition of legality that we used in the geometric case by insisting that 
 we do not count copies of $\sigma$ when counting the length. 
 Choose an improved relative train track map $f:G\to G$ representing $\phi$ and let $G_{r-1}$ be the filtration element that realizes 
 $\mathcal{F}$. Then $H_r$ is the EG strata associated with $\Lambda^\pm_\phi$ and the circuit realizing $[\sigma]$ is the unique indivisible Nielsen 
 path of height $r$.
 
 For any circuit $\alpha$ in $G$, define ${|\alpha|}_{H_r^\sigma}$ to be the $H_r$-edge length of $\alpha$ relative to the realization of 
 $[\sigma]$ in $G$, i.e. the length obtained by counting the edges of $H_r$ but not counting copies of the closed indivisible Nielsen path inside  $\alpha$ 
 that represents $[\sigma]$.
 
 \begin{definition}
  For any circuit $\alpha$ in $G$, the $H_{r}$-legality of $\alpha$ is defined as the ratio 
  $${LEG}_{H_r}(\alpha):= \frac{\text{sum of lengths of generic leaf segments of $\Lambda^+_\phi$ in $\alpha\cap H_r$ of length} \geq C}{{|\alpha|}_{H_r^\sigma}}$$
 if ${|\alpha|}_{H_r^\sigma}\neq 0$. Otherwise, if ${|\alpha|}_{H_r^\sigma}=0$, define ${LEG}_{H_r}(\alpha)=0$.
 
 \end{definition}
 
 In this setup consider the collection of subgroups $\{F^1, F^2,..., F^k, \langle \sigma\rangle\}$.
 Since the nonattracting subgroup is a malnormal 
 system, this collection of subgroups is also malnormal collection of subgroups. Also since $\sigma$ 
 is a primitive element then each subgroup of the collection is quasiconvex. 
 
 Consider the following short exact sequence of groups 
 $$1 \to \F \to \Gamma \to \langle \phi \rangle \to 1$$
 
 We perform a partial electrocution of $\Gamma$ and $\F$ by coning-off the collection of subgroups 
 $\{F^i\}\cup\langle\sigma\rangle$. Then $\F$ is (strongly) relatively hyperbolic and the Cayley graph of $\langle\phi\rangle$
 being a tree gives us a tree of strongly relatively hyperbolic spaces with vertex spaces being identified with cosets 
 of $\F$. Thus we may look at the Cayley graph of $\Gamma$ as a tree of relatively hyperbolic spaces.
 Recall the setup explained before the statement of the strong combination theorem \ref{hallway}.
 
 In this situation the following simple modifications to the results proved for the nongeometric case can be made:
 
 \begin{enumerate}
  \item Conclusion of Lemma \ref{legality} and Lemma \ref{flare} are true if we consider  circuits
  $\alpha$ which are not carried by $\mathcal{F}\cup \{[\langle \sigma \rangle]\}$.
  \item Conclusion of Lemma \ref{comparison} is true if we replace  
  ${|\alpha|}_{H_r}$ with ${|\alpha|}_{H_r^\sigma}$ in the statement. In particular the inequality is true for any conjugacy class not carried by 
  $\mathcal{F}\cup \{[\langle \sigma \rangle]\}$.
  The relevant modification in the proof is made by considering edges, 
  paths and circuits not carried by the path system $\langle \mathcal{Z}, \hat{\sigma}\rangle$ instead of $\mathcal{Z}$. 
  \item Lemma \ref{conjflare} is true for all conjugacy classes not carried by $\mathcal{F}\cup \{[\langle \sigma \rangle]\}$ 
  in the statement.
  \item Proposition \ref{strictflare}  is true for all words in $\F\setminus \{\cup F^i\cup \sigma\}$.
 \end{enumerate}

  Thus arguing exactly as we did for the nongeometric case, we have the following theorem:
 
 \begin{theorem}\label{geom}
  Consider  $\phi\in\out$  and a free factor system 
 $\mathcal{F}=\{[F^1], [F^2],...., [F^k]\}$ 
 such that $\phi$ is fully irreducible relative to $\mathcal{F}$ and geometric above $\mathcal{F}$.
 Then the extension group $\Gamma$ in the short exact sequence 
 $$1 \to \F \to \Gamma \to \langle \phi \rangle \to 1$$ is strongly hyperbolic 
 relative to the collection of subgroups $\{F^i\rtimes \Phi_i\}$ and $\langle\sigma\rangle\rtimes \Phi_\sigma$
  where $\Phi_i$ is a lift that  preserves $F^i$ and $\Phi_\sigma$ is a lift that fixes $\sigma$.
  
\end{theorem}
If we take $\mathcal{F}$ to be empty, then we get the following corollary:

\begin{corollary}
 
 For every fully irreducible and geometric $\phi\in\out$ the extension groups 
 $\Gamma$ in the short exact sequence $$1 \to \F \to \Gamma \to \langle\phi\rangle \to 1 $$
 is strongly hyperbolic relative to the subgroup $\langle\sigma\rangle\rtimes_\Phi \mathbb{Z}$ where 
 $\Phi$ is a lift of $\phi$ that fixes $\sigma$.
\end{corollary}

The more general case when we have two $\phi, \psi\in\out$  such that both are 
geometric above $\mathcal{F}$ has to be handled with a little more care. If it so happens that 
both these automorphisms fix the same conjugacy class $\sigma$, where $\sigma$ is as described in the 
beginning of this section, we get the conclusion of Mosher's 3-of-4 lemma 
as in Proposition \ref{34} for every conjugacy class not carried by $\mathcal{F}\cup \{[\langle\sigma\rangle]\}$ (for part 1) and 
using it one shows the flaring property of part (2) for every word not in $\F\setminus \{\cup F^i\cup \sigma\}$.
Thus an argument similar to the nongeometric case \ref{nongeoext} gives us the following theorem:

\begin{theorem}\label{geoext}
 Suppose $\phi,\psi\in\out$ are rotationless and 
  $\mathcal{F}=\{[F^1], [F^2],..., [F^k]\}$ be a $\phi, \psi-$invariant free factor system such that 
  $\mathcal{F}\sqsubset\{[\F]\}$  
  is a multi-edge extension and $\phi, \psi$ are fully irreducible relative to $\mathcal{F}$ 
  and both are geometric above $\mathcal{F}$ and pairwise independent relative to $\mathcal{F}$.
  Also assume that they fix the same conjugacy class $\sigma$ (described above).
 Then the extension group $\Gamma$ in the short exact sequence 
  $$ 1 \to \F \to \Gamma \to Q \to 1 $$ is strongly  hyperbolic relative to the 
  collection of subgroups 
  $\{F^i \rtimes \widehat{Q_i} \}$ and $\langle\sigma\rangle\rtimes \widehat{Q_\sigma}$, where $Q=\langle \phi^m, \psi^n \rangle$ is the free group in the conclusion of 
  Corollary \ref{rfi} and $\widehat{Q_i}$  is a lift that preserves $F^i$ and 
  $\widehat{Q_\sigma}$ is a lift that fixes $\sigma$.
  
\end{theorem}
\begin{proof}
 
 Recall that we are working with an outer automorphisms $\phi,\psi$ which are fully irreducible 
  relative to a free factor system $\mathcal{F}$, where $\Lambda^\pm_\phi, \Lambda^\pm_\psi$ 
  are geometric above $\mathcal{F}$ and $\mathcal{A}_{na}(\Lambda^\pm_\psi)=\mathcal{A}_{na}(\Lambda^\pm_\phi)=\mathcal{F}\cup \{[\langle\sigma\rangle]\}$. 
  We have chosen some representative $F^i$ for each component $[F^i]$ of $\mathcal{F}$ and a representative $\sigma$ of 
  $[\sigma]$. 
  Then we performed a partial electrocution of the extension group $\Gamma$ $$1\to\F\to\Gamma\to Q\to 1 $$
  with respect the collection of 
  subgroups $\{F^i\}\cup\langle\sigma\rangle$ and denoted it by $(\widehat{\Gamma}, {|\cdot|}_{el})$. We also performed an electrocution of 
  $\F$ with respect to the collection $\{F^i\}\cup\langle\sigma\rangle$, denoted by $\widehat{\F}$, and since $\mathcal{A}_{na}(\Lambda^\pm_\phi)$
  is a malnormal subgroup system, $\F$ is (strongly) relatively hyperbolic with respect to the collection $\{F^i\}\cup\langle\sigma\rangle$. 
  The Cayley graph of the quotient group $Q$ 
  being a tree, gives us a tree of (strongly) relatively hyperbolic spaces with vertex spaces being identified 
  with cosets of $\F$. Thus we may regard the Cayley graph of $\Gamma$ as a tree of (strongly) relatively hyperbolic 
  spaces and then $\widehat{\Gamma}$ is the induced tree of coned-off spaces in the statement of the Mj-Reeves strong combination theorem \ref{hallway}. 
  Proposition \ref{34} proves that the hallway flare condition and the cone-bounded hallways strictly flare condition 
  are satisfied for this induced tree of coned-off spaces.  Lemma \ref{hallway} then gives us the desired conclusion.
 
\end{proof}

As a corollary of this when we take $\mathcal{F}=\emptyset$, we recover the case for 
surface group with punctures, which was proved in \cite[Theorem 4.9]{MjR-08}. 

Finally as a concluding remark we would like to point out that for the groups constructed 
in \cite[Corollary 6.1, item (1)]{self1} we can use of Theorem \ref{nongeoext} with $\mathcal{F}=\emptyset$ to conclude that 
the extension given by that free group is a hyperbolic extension. But no conclusion 
can be drawn about the other two types of groups constructed in that corollary, namely 
in item (3) and (4) of Corollary 6.1 using the results we have developed here. 
In fact the Mj-Reeves strong combination theorem cannot be applied in that case to deduce 
relative hyperbolicity. Hence we can ask the question whether the extension defined by 
those groups are strongly  relatively 
hyperbolic relative to any \textit{finite} collection of subgroups ?

\noindent\rule[0.5ex]{\linewidth}{1pt}
Address: Department of Mathematical Sciences, IISER Mohali, Punjab, India.\\
Contact: \href{mailto:pritam@scarletmail.rutgers.edu}{pritam@scarletmail.rutgers.edu}
%
\bibliographystyle{plainnat}
\def\bibfont{\footnotesize}
\bibliography{biblo}

\end{document}